\documentclass[11pt,oneside,reqno]{amsart}
\usepackage[english]{babel}
\usepackage[colorlinks,citecolor=green,linkcolor=red]{hyperref}
\usepackage{amsthm}
\usepackage{color}
\usepackage{mathrsfs}
\usepackage{amsmath}
\usepackage{mathtools}
\usepackage{amsfonts}
\usepackage{amssymb}
\usepackage{bm}
\usepackage{physics}
\usepackage{enumitem}
\usepackage[margin=0.7in]{geometry}
\usepackage{cleveref}
\usepackage{esint}
\usepackage{nicefrac}
\usepackage{yfonts}
\usepackage{dirtytalk}
\usepackage{verbatim}
\usepackage{xfrac}
\usepackage{stmaryrd,mathrsfs,bm,amsthm,mathtools,yfonts,amssymb,color}

\newtheorem{thm}{Theorem}
\newtheorem*{thm*}{Theorem}
\newtheorem{lem}[thm]{Lemma}
\newtheorem{prop}[thm]{Proposition}
\newtheorem{cor}[thm]{Corollary}
\newtheorem{con}[thm]{Conjecture}

\theoremstyle{definition}
\newtheorem{defn}[thm]{Definition}

\theoremstyle{remark}

\numberwithin{equation}{section}
\numberwithin{thm}{section}

\newcommand{\defeq}{\coloneqq}
\newcommand{\eqdef}{\eqqcolon}

\def\XXint#1#2#3{{\setbox0=\hbox{$#1{#2#3}{\int}$}
		\vcenter{\hbox{$#2#3$}}\kern-.5\wd0}}

\newcommand{\RR}{\mathbb{R}}
\newcommand{\NN}{\mathbb{N}}

\newcommand\supp{{\rm spt}\,}

\newcommand\res{\mathop{\hbox{\vrule height 7pt width .3pt depth 0pt
			\vrule height .3pt width 5pt depth 0pt}}\nolimits}
\newcommand{\mres}{\res}

\newcommand{\bG}{\mathbf{G}}

\newcommand{\p}{{\mathbf{p}}}

\newcommand{\bE}{{\mathbf{E}}}

\newcommand{\bB}{{\mathbf{B}}}
\newcommand{\bC}{{\mathbf{C}}}

\newcommand{\be}{{\mathbf{e}}}

\newcommand\N{{\mathbb N}}

\newcommand\R{{\mathbb R}}

\newcommand{\eps}{{\varepsilon}}
\renewcommand{\epsilon}{{\varepsilon}}

\newcommand{\bmax}{\mathbf{m}}

\newcommand{\diam}{{\rm {diam}}}
\newcommand{\dist}{{\rm {dist}}}

\newcommand{\cG}{{\mathcal{G}}}

\newcommand{\Iqs}{{\mathcal{A}}_Q(\R^{n})}
\newcommand{\Iq}{{\mathcal{A}}_Q}
\def\a#1{\llbracket{#1}\rrbracket}

\newcommand{\D}{\textup{Dir}}

\newcommand{\etaa}{{\bm{\eta}}}

\newcommand\B{\bB}

\makeatletter 
\def\@tocline#1#2#3#4#5#6#7{\relax
	\ifnum #1>\c@tocdepth 
	\else
	\par \addpenalty\@secpenalty\addvspace{#2}%
	\begingroup \hyphenpenalty\@M
	\@ifempty{#4}{%
		\@tempdima\csname r@tocindent\number#1\endcsname\relax
	}{%
		\@tempdima#4\relax
	}%
	\parindent\z@ \leftskip#3\relax \advance\leftskip\@tempdima\relax
	\rightskip\@pnumwidth plus4em \parfillskip-\@pnumwidth
	#5\leavevmode\hskip-\@tempdima
	\ifcase #1
	\or\or \hskip 1em \or \hskip 2em \else \hskip 3em \fi%
	#6\nobreak\relax
	\dotfill\hbox to\@pnumwidth{\@tocpagenum{#7}}\par
	\nobreak
	\endgroup
	\fi}
\makeatother

\title[Stationary varifolds]{Remarks and Conjectures on Stationary Varifolds}
\author[C. Brena]{Camillo Brena}
\address{School of Mathematics, Institute for Advanced Study, 1 Einstein Dr., Princeton NJ 08540, USA}
\email{cbrena@ias.edu}
\author[S. Decio]{Stefano Decio}
\address{Department of Mathematics, ETH Z\"urich, R\"amistrasse 101, 8092 Z\"urich, Switzerland}
\email{stefano.decio@math.ethz.ch}
\author[C. De Lellis]{Camillo De Lellis}
\address{School of Mathematics, Institute for Advanced Study, 1 Einstein Dr., Princeton NJ 08540, USA}
\email{camillo.delellis@ias.edu}

\begin{document}
	\begin{abstract}		
In this paper, we revisit some known results about stationary varifolds using simpler arguments. In particular, we obtain the height bound and the Lipschitz approximation along with its estimates, and as a consequence, the excess decay
		\end{abstract}
        \maketitle
	\tableofcontents
 		  \section{Introduction}

The regularity theory of stationary varifolds was pioneered by Bill Allard in his two classical papers \cite{AllardFirst} and \cite{AllardBoundary}, which are concerned, respectively, with the interior and boundary regularity. In this note we will focus our interest on integer rectifiable varifolds of dimension $m$ in a open subset $U$ of the Euclidean space $\mathbb R^{m+n}$. These objects can be identified with Radon measures $\|V\|=\Theta \mathcal{H}^m \res E$ on $U$, where $E$ is an $m$-dimensional rectifiable set, $\mathcal{H}^m$ denotes the Hausdorff $m$-dimensional measure, and $\Theta$ is a Borel function taking positive integer values ($\mathcal{H}^m$-a.e.). This means that 
\begin{equation}
    \int \varphi\,\dd\|V\| = \int_E \Theta (x) \varphi (x)\,\dd\mathcal{H}^m (x) \qquad \text{for every } \varphi \in C_c (U)\, .
\end{equation}
Hence we will always understand the varifold as a pair $(E, \Theta)$, up to $\mathcal{H}^m$-null sets, or as a measure $\|V\|$ of the form above. We will assume that $V$ is stationary in $U$: this means that, for every given $X\in C^\infty_c (U, \mathbb R^{m+n})$, if we let $\Phi_t$ be the one-parameter family of diffeomorphisms of $U$ generated by $X$, then
\begin{equation}
\delta V (X) \defeq \left.\frac{d}{dt}\right|_{t=0} \|(\Phi_t)_\sharp V\| (U) = 0\, .
\end{equation}
Here $\psi_\sharp V$ denotes the varifold $(\psi (E), \Theta\circ \psi^{-1})$ when $\psi$ is a $C^1$ diffeomorphism.
It follows from Allard's monotonicity formula that, since $V$ is integral and stationary, we can assume that, without loss of generality, $E$ is $\supp(\|V\|)\cap U$ and $\Theta$ is given pointwise by the upper semicontinuous function
\begin{equation}\label{e:density}
\Theta (V, x) = \lim_{r\searrow 0} \frac{\|V\| (B_r (x))}{\omega_m r^m}\, .
\end{equation}
This gives a ``canonical pair'' and rids us of any tedious discussion of $\mathcal{H}^m$-null sets.
Finally, we will be interested in the interior regularity theory of such objects. More precisely, using the notation $\supp(V)$ for $\supp (\|V\|)$ (a convention which will be adopted in the rest of this note), a point $x\in \supp(V)\cap U$ will be called regular if there is a neighborhood $U'$ of $x$ with the property that $\supp(V)\cap U'$ is a $C^1$ submanifold of $U'$ (without boundary in $U'$). In this case Allard's constancy theorem implies that $\Theta (V, \,\cdot\,)$ is locally constant on the latter submanifold. The interior singular set is the complement of the regular points in $\supp(V)\cap U$.

\medskip

Under the above assumptions, a corollary of Allard's classical interior regularity theory developed in \cite{AllardFirst} is that the regular points of $V$ form a (relatively open) dense set in $\supp(V) \cap U$. While it is well known that stationary varifolds can form (even very complicated) singularities, there are no examples in which the singular set has dimension bigger than $m-1$ and strong indications that in fact singularities cannot be larger, see for instance the recent result \cite{HS}. We recall therefore the following well-known conjecture.
\begin{con}\label{c:codim-uno}
The interior singular set of a stationary integer rectifiable $m$-dimensional varifold has codimension not smaller  than $1$.
\end{con}
On the other hand, though 52 years have passed since \cite{AllardFirst}, even the following more modest very plausible conjecture is still open.

\begin{con}\label{c:measure-zero}
The interior singular set of a stationary integer rectifiable $m$-dimensional varifold is $\mathcal{H}^m$-null.
\end{con}

The above statement is correct (and in fact the singular set turns out to be much smaller) under some stronger variational assumption, though it is still open even for stable general varifolds in codimension one; but rather than surveying the many results in the literature we refer the reader to \cite{DeLellis-ICM}. On the other hand the statement is false for slight generalizations of the conjecture: for instance if the varifold is assumed to have bounded mean curvature. In this case the theory of Allard still applies, and hence the singular set has to be meager in the topological sense, but at the same time both Allard in \cite{AllardFirst} and Brakke a few years later in \cite{Brakke} show examples for which the singular set has positive Hausdorff measure.

Given how hard the above problem has proved to be, we propose in this paper a number of more modest goals, each of which retains at least a portion of the most challenging aspects of proving Conjecture \ref{c:measure-zero}. These statements, if taken all together, provide a path to Conjecture \ref{c:measure-zero}, but we are not claiming that ours is a ``new program'': in fact our statements emerge quite naturally from the various results available in the literature and any expert in the subject could easily come up with them. We however think that it is worth to point them out and also that there are appropriate tools to prove at least some of them. In particular another purpose of this note is to prepare some technical ground for the forthcoming work \cite{BDF}. In passing we will give some alternative shorter (and in our opinion more transparent) proofs of known results and will discuss some ``folklore knowledge'' about their optimality.

\subsection{Conjectures} From now on we assume that the reader has some familiarity with the regularity theory for minimal surfaces, our notation and terminology is very similar to that of \cite{Simon} and \cite{DelellisAllard}. First of all, we recall the following two consequences of \cite{AllardFirst}:
\begin{itemize}
\item[(a)] $\Theta(V ,x) \geq 1$ at every $x\in \supp(V)$, and there is $\varepsilon>0$ such that, if $\Theta(V,x)< 1+\varepsilon_0$, then $x$ is a regular point;
\item[(b)] If $\Theta(V,x) < Q (1+\varepsilon_0)$ and there is a neighborhood $U'$ of $x$ in which $\Theta(V,\,\cdot\,) \geq Q$ $\mathcal{H}^m$-a.e., then $x$ is a regular point.
\end{itemize}
On the other hand, from basic measure theoretic facts (and elementary properties of rectifiable sets), at $\mathcal{H}^m$-a.e.\ $x_0$ there is a unique tangent cone to $V$ which is ``flat'', namely supported in the $m$-dimensional plane $T_{x_0} V$ which is the approximate tangent to the rectifiable set $\supp(V)$ at $x_0$. This means that, if we denote by $\tau_{x,r}$ the dilation maps 
\begin{equation}
    \tau_{x,r} (y) = \frac{y-x}{r}\, , 
\end{equation}
then $\|(\tau_{x_0,r})_\sharp V\|$ converges (locally in the weak$^\star$ topology) to the Radon measure $\Theta(V,x_0) \mathcal{H}^m \res T_{x_0} V$. It follows from Allard's compactness for stationary integer rectifiable varifolds that $\Theta (V,x_0)$ is an integer.

Conjecture \ref{c:measure-zero} would then be implied by the following.

\begin{con}\label{c:flat=regular}
If $V$ has a unique flat tangent cone at $x_0$ and 
\begin{equation}
\lim_{r\searrow 0} \frac{\|V\| (\{\Theta(V,\,\cdot\,)\neq \Theta (V, x_0)\}\cap \mathbf{B}_r (x_0))}{r^m} = 0\, ,
\end{equation}
then $x_0$ is a regular point.
\end{con}

A key point of Allard's theory is that the latter conjecture is indeed correct under the additional assumption that $\Theta(V,\,\cdot\,)\geq \Theta (V,x_0)$ a.e.\ in a neighborhood of $x_0$. Note that, by the upper semicontinuity of $\Theta(V,\,\cdot\,)$ and the fact that it is integer-valued $\mathcal{H}^m$-a.e., the above is equivalent to say that $\Theta(V,\,\cdot\,)$ equals $\Theta (V,x_0)$ a.e.\ in some neighborhood of $x_0$. We could then locally factor this constant value and reduce to assuming that $\Theta(V,\,\cdot\,)=1$ a.e..

The great challenge in Conjecture \ref{c:flat=regular} lies therefore in the possibility that $x_0$ is the accumulation of small regions of $\supp(V)$ with positive Hausdorff measure where the density of the varifold is at most $Q-1$ (a.e.). This is in fact precisely what happens in the examples of Allard and Brakke under the more general assumption that the mean curvature of $V$ is bounded rather than $0$. In these examples, however, such bad ``singular points" have a rather high order of contact with a smooth surface. If we imagine that this is a general fact, for stationary varifold at a ``bad'' point we could hope that there is a classical minimal surface with a high order a contact with the varifold. We could therefore advance the following conjecture:

\begin{con}\label{c:flat-high-order}
Assume $V$ and $x_0$ satisfy the assumptions of Conjecture \ref{c:flat=regular}. Then there is a smooth classical minimal $m$-dimensional graph $\mathcal{M}$ in some neighborhood of $x_0$ with the property that 
\begin{equation}\label{e:order-of-contact}
\int_{\mathbf{B}_r (x_0)} \dist (x, \mathcal{M})^2 \,\dd\|V\| (x) = o (r^N)\qquad \mbox{for every $N\in \mathbb N$\,.}
\end{equation}
\end{con}

If this were correct, then Conjecture \ref{c:flat=regular} would essentially turn into a suitable unique-continuation problem for stationary varifolds. This would be no surprise to the experts. Indeed before Almgren's celebrated Big regularity paper \cite{Almgren00} the regularity theory of area-minimizing currents in codimension higher than $1$ was stuck essentially at the same point as are stationary varifolds now. When Almgren was famously able to solve the problem, part of his approach translated into a new way to prove unique continuation for elliptic PDEs, see e.g. the survey \cite{DeLellis-survey}. Following this path, the second big step to Conjecture \ref{c:flat=regular} would then be

\begin{con}\label{c:unique-continuation}
Assume $V$, $x_0$, and $\mathcal{M}$ are, respectively, a stationary varifold, a point $x_0\in \supp(V)\cap U$, and a classical smooth minimal $m$-dimensional surface such that \eqref{e:order-of-contact} holds. Then $\supp(V)\subseteq \mathcal{M}$ in some neighborhood of $x_0$.
\end{con}

The latter conjecture is open (and interesting) even in the very special case that $\mathcal{M}$ is just an $m$-dimensional plane $\pi_0$.  This particular case, as pointed out to us by Federico Franceschini, can be recasted as another unique continuation problem. As it is well known, coordinate functions are harmonic on classical minimal surfaces. This fact can be generalized to varifolds.

\begin{defn}
Given a $C^1$ function $\varphi: U \to \mathbb R$ and an $m$-dimensional varifold $V$ stationary in $U\subseteq \mathbb R^{m+n}$ we define, for $\mathcal{H}^m$-a.e.\ $x\in \supp (V)\cap U$, the vector $\nabla_{TV} \varphi (x)$ as the projection  of $\nabla \varphi$ on the approximate tangent space to $\supp (V)$ at $x$. A function $h \in C^1 (U)$ is then called {\em harmonic} on $V$ if 
\[
\int \nabla_{TV} h \cdot \nabla_{TV} \varphi \, d\|V\|  = 0 \qquad \text{for every } \varphi\in C^1_c (U)\, .
\]
\end{defn}

In particular, if $h$ is a coordinate function on $\mathbb R^{m+n}$, $e$ the unit vector $\nabla h$, and $\varphi\in C^1_c (U)$, a straightforward computation using the first variation formula shows
\[
0= \delta V (\varphi e) = \int \nabla_{TV} h (x) \cdot \nabla_{TV} \varphi (x)\, d\|V\| (x)\, ,
\]
so that the harmonicity of $h$ follows from the stationarity of $V$.
Assume therefore that $V$, $x_0$ and $\mathcal{M}$ are as in Conjecture \ref{c:unique-continuation}, with the additional information that $\mathcal{M}$ is a plane $\pi_0$. Then for all coordinates functions $h$ which vanish on $\pi_0$ we know that $\int_{B_r (x_0)} h^2 d\|V\|= o (r^N)$ for every $N\in \mathbb N$, while the inclusion $\supp (V)\subseteq \mathcal{M} = \pi_0$ is equivalent to know that any such coordinate function $h$ vanishes identically on $\supp (V)$. This motivates the following 

\begin{con}\label{c:harmonic}
Assume $V$ is a stationary $m$-dimensional varifold in $U\subseteq \mathbb R^{m+n}$, $h$ an harmonic function on $V$, and $x_0\in \supp (V)\cap U$ a point such that 
\[
\int_{B_r (x_0)} h^2 d\|V\|= o (r^N)\qquad \forall N\in \mathbb N\, .
\]
Then $h$ vanishes identically on $\supp (V)\cap B_r (x_0)$ for some $r\in(0, \dist (x_0, \partial U))$.
\end{con}

In this and the subsequent note \cite{BDF} we will be concerned with partial results towards Conjecture \ref{c:flat-high-order}.

\subsection{Rectifiability} Take an $m$-dimensional  integer rectifiable varifold $V$, assume that $V$ is stationary in an open subset $U\subseteq\RR^{m+n}$. By its very definition, the varifold is $C^1$-rectifiable, in the sense that $\supp(V)$ is a $C^1$-rectifiable subset: $\supp (V)$ can be covered with countably many $C^1$ submanifold with the exception of an $\mathcal{H}^m$-null set. As pointed out by Brakke in \cite{Brakke}, in combination with Almgren's theory of multivalued function, Allard's approach already implies $C^{1,\alpha}$ rectifiability for every $\alpha<1$: the submanifolds of the covering can actually be chosen to be $C^{1,\alpha}$. In the rest of this note we present a more transparent derivation of Brakke's theorem. A much more interesting development, due to Menne in \cite{MenneJGA}, is the $C^2$-rectifiability. In fact Menne's theorem applies to varifolds with bounded mean curvature and it is thus an optimal statement in his context.  In \cite{BDF} the first and second author together with Federico Franceschini prove that stationary varifolds are indeed $C^\infty$ rectifiable. As a byproduct of the proof of \cite{BDF}, we can conclude the following.
\begin{thm}\notag
Assume $V$ and $x_0$ satisfy the assumptions of Conjecture \ref{c:flat=regular}. Then there is a smooth classical (not necessarily minimal) $m$-dimensional graph $\mathcal{M}$ in some neighborhood of $x_0$ with the property that 
\begin{equation}\notag
\int_{\mathbf{B}_r (x_0)} \dist (x, \mathcal{M})^2 \,\dd\|V\| (x) = o (r^N)\qquad \mbox{for every $N\in \mathbb N$\,.}
\end{equation}
\end{thm}
This can be seen as an answer to a  weaker (as  we are not able to prove that $\mathcal{M}$  is minimal!) version of Conjecture \ref{c:flat-high-order}.

\subsection{Main statements}
We will fix a reference plane $\pi_0\defeq\RR^m\times \{0\}$ in $\mathbb R^{m+n}$. We denote by $\bB_r(x)=\{y:|y-x|<r\}$ the  balls in $\RR^{m\times n}$ and by $B_r(x)$ the disks $\bB_r (x)\cap \pi_0$ and we omit the center $x$ when it is the origin. With an abuse of notation, we are going to use  $B_r(x)$ also for disks in $\RR^n$, however, this will not cause confusion as the nature of each disk will be clear from the context. 
We write $\bar \bB_r(x)$ to denote the closed balls, i.e.\ $\{y:|y-x|\le r\}$, and we use a similar notation for the disks. 
In addition $\bC_r{(x,\pi)}=\{y\in\RR^n: |\p_{\pi}(y-x)|<r\}$ will denote the cylinder of radius $r$ with ($n$-dimensional) axis going through center $x$ and perpendicular to $\pi$ (if $x=0$ and $\pi=\pi_0$, we simply write $\bC_r$).   Here (and in the rest of the notes) $\p_\pi$ denotes the orthogonal projection on $\pi$, while the distance between two (unoriented) planes $\pi$ and $\pi'$ will be measured through $|\p_{\pi}-\p_{\pi'}|$, the Hilbert-Schmidt norm of the difference between the respective orthogonal projections.
    
Unless specified, the $m$-dimensional integral varifold $V$ under consideration is assumed to be stationary in any open set appearing in our arguments (typically some cylinder). As above $\|V\|$ will denote the corresponding Radon measure, while the density $\Theta$ will be defined at every point by \eqref{e:density}, but the dependence on $V$ will be dropped when clear from the context. 

We define the cylindrical and spherical excesses of $V$ in cylinders and balls as 
	\begin{align}
		\bE(V,\bC_r(x,\pi'),\pi)&\defeq\frac{1}{\omega_m r^m}\int_{\bC_r(x,\pi')} |\p_{T V}-\p_{\pi}|^2\,\dd\|V\|\,,\\
         \bE(V,\bB_r(x))&\defeq \inf_{\pi}\frac{1}{\omega_m r^m}\int_{\bB_r(x)} |\p_{T V}-\p_{\pi}|^2\,\dd\|V\|\,.
	\end{align}
	and we will use the shorthand notation $\bE_r\defeq \bE(V,\bC_r,\pi_0)$. 
    
The letter $Q$ is used for an important (positive) integer parameter, which can be understood as the number of times that the varifold (approximately) covers the base of the cylinder. We stress that this holds only in the approximate sense. For instance in the case of a classical $2$-dimensional catenoid which in a given $3$-dimensional cylinder $\bC_r (x, \pi)$ is very close to its planar cross section $B_r (x, \pi)$ and is invariant under rotation around the axis of the cylinder, the parameter $Q$ will be $2$, even though there will always be some (possibly small) region $B_\rho (x, \pi)$ with the property that $\p_{\pi}^{-1} (B_\rho (x, \pi)) \cap \supp (V) = \emptyset$. The parameters $m,n,Q$ will be called ``geometric parameters''.

\medskip

  Our first theorem is an $L^\infty$ height bound: the support of the stationary varifold in a cylinder is contained in at most $Q$ strips of width comparable to a suitable function of the excess in a bigger cylinder. Moreover, if the varifold has a point of density $Q$, then we can choose a single strip, with width comparable to the square root of the excess in the bigger cylinder. 
  For what concerns the first estimate, \eqref{bvefdasdue}, the dependence in terms of the excess is  optimal at least in dimension $2$, and this can be shown  using catenoids (we refer to the next section). Instead,  the second estimate, \eqref{heighatdensity}, cannot be improved even for classical minimal graphs, irrespectively of the dimension and codimension.
  
    \begin{thm}[Height bound]\label{heightdue} Let $V$ be stationary in $\bC_1$ and $r\in (0,1)$.
Assume that $\bE_1$ is small enough (depending upon $ r$ and the geometric parameters) and that $\frac{\|V\|(\bC_1)}{\omega_m}\le Q+\textstyle{\frac{1}{2}}$. 
Then there exist a constant $C= C(r,m,n,Q)$ and points  $y_1,\dots,y_Q\in \pi_0^\perp$ (not necessarily distinct) such that 
		\begin{equation}\label{bvefdasdue}
			\supp(V)\cap\bC_{ r}\subseteq \bigcup_{h=1,\dots,Q} \pi_0\times \bar B_{C(r)|\log\bE_1|^{1-1/m}\bE_1^{1/m}}(y_h, \pi_0^\perp)\,.
		\end{equation}
If in addition $\Theta(0)=Q$, then 
   \begin{equation}\label{heighatdensity}
      \supp(V)\cap \bC_r\subseteq \pi_0\times \bar B_{C_0\bE_1^{1/2}}(0, \pi_0^\perp)\,.
  \end{equation}
\end{thm}
From now on in order to keep our notation lighter, we will identify $\pi_0$ with the factor $\mathbb R^m$ and $\pi_0^\perp$ with the factor $\mathbb R^n$ in the product space $\mathbb R^{m+n}$. The second theorem
is a Lipschitz approximation result: given a stationary varifold, there exists a Lipschitz $Q$-valued function whose graph coincides  with the varifold up to a set of size comparable to the excess. In order to state it we introduce the following useful object, while for the notation on multivalued functions we follow \cite{DSq}.

\begin{defn}\label{defnmax}
	We define the ``non-centered'' maximal function for $\bE$ as
	\begin{equation}
		\bmax\be(x)\defeq \sup_{x\in\bC_s(y)\subseteq\bC_4}\bE(V,\bC_s(y))\,.
	\end{equation}
\end{defn}

 \begin{thm}[Lipschitz approximation]\label{Lipapproxstrong} Let $V$ be stationary in $\bC_4$.
	Assume that $\bE_4$ is small enough (depending upon the geometric parameters), that $\frac{\|V\|(\bC_4)}{\omega_m 4^m}\le Q+1$, and that
	\begin{equation}
 \frac{\|V\|(\bC_{3})}{\omega_m 3^m}  \in(Q-\textstyle{\frac{1}{2}},Q+\textstyle{\frac{1}{2}})\, .
	\end{equation} 
	Set, for $\lambda \in (0,1)$ small enough (depending upon the geometric parameters),
	\begin{equation}
		K_\lambda\defeq \{x\in B_1:\bmax\be (x)\le \lambda\}\,.
	\end{equation}
Then, there exists a $Q$-valued Lipschitz function $f:B_1\rightarrow\Iqs$ such that  
\begin{equation}\label{Lipdcsaim}
		f\text{ is }C_0|\log\lambda|^{1-1/m}\lambda^{1/m}\text{-Lipschitz}\,.
	\end{equation}
In addition:
\begin{itemize}
\item[(i)] For every $x\in K_\lambda$ and $x\ni B_{s}(y)\subseteq B_3$  
\begin{equation}\label{vadfscas}
\frac{\|V\|(\bC_{s}(y))}{\omega_m s^m}\in (Q-\textstyle{\frac{1}{2}},Q+\textstyle{\frac{1}{2}})\, .
\end{equation}
\item[(ii)] For every $x$ and $s$ as in (ii), if $f(x)=\sum_i Q_i\a{p_i}$, then $\supp(V)\cap(\{x\}\times\RR^n) = \bigcup_i (x,p_i)$ and
\begin{equation}\label{brvfeads}
\supp(V)\cap \bC_s(y)\subseteq  B_s(y)\times \bigcup_i B_{C_0\lambda^{1/(2m)} s}(p_i)\,.
\end{equation}
\item[(iii)] The following estimate holds 
	\begin{equation}\label{measureestimatem}
		|B_1\setminus K_\lambda|+\|V\|(\bC_1\setminus (K_\lambda\times\RR^n))\le C_0\frac{1}{\lambda}\bE_4\,.
	\end{equation}
\item[(iv)] For every $(x,y)\in (K_\lambda\times \RR^n)\cap\supp(V)$, if we write as above $f(x)=\sum_i Q_i\a{p_i}$, then
	\begin{equation}\label{densdvacs}
		\Theta(x,y)=\sum_{i:p_i=y}Q_i\,.
	\end{equation}
\item[(v)] Finally, if $A>0$ is such that  $\supp(V)\cap\bC_1\subseteq \RR^m\times B_{A}(0)$, then 
    \begin{equation}\label{vefadcsxac}
		\sup_{x\in B_1} \cG(f(x),Q\a{0})\le C_0 A\,.
	\end{equation}
\end{itemize}
\end{thm}

We then use the above Lipschitz approximation result in the classical strategy of Allard to prove excess decay. More precisely, if at some point $x_0$
\begin{itemize}
    \item the density is $Q$,
    \item the portion of the points of density smaller than $Q$ is small at every scale around $x_0$, 
    \item and the spherical excess at some initial scale is sufficiently small, 
\end{itemize}
then the spherical excess decays almost quadratically at $x_0$.

\begin{thm}[Excess decay]\label{decay} For every $\delta\in (0,1)$ there is a threshold $\epsilon = \varepsilon (\delta, Q, m, n) >0$ and a constant $C= C (\delta, Q,m,n)$ with the following property. If $V$ is stationary in $\bB_1$ and the following conditions hold 
 \begin{gather}
     \bE(V,\bB_1)<\epsilon\,,\\
     \frac{\mathcal{H}^m(\{\Theta<Q\}\cap \supp (V) \cap \bB_r)}{\omega_m r^m}<\epsilon\qquad\text{for every $r\in (0,1)$}\,,\\
     \Theta(0)=Q\quad\text{and}\quad
     \frac{\|V\|(\bB_1)}{\omega_m}<Q+\epsilon\,,\label{vrfecdsccas}
 \end{gather}
	then
	\begin{equation}
		\bE(V,\bB_r)\le C(\delta) r^{2-2\delta}\bE(V,\bB_1)\qquad\text{for every $r\in (0,1)$}\,.
	\end{equation}
\end{thm}

From the above theorem we immediately conclude the $C^{1,1-\delta}$ rectifiability of $\supp (V)$.

\begin{cor}
    Let $V$ be stationary in $\bB_1$.  Then, for any $\delta\in (0,1)$, 
    \begin{equation}
        \lim_{r\searrow 0} r^{2-2\delta}\bE(V,\bB_r(x))=0\qquad\text{for $\|V\|$-a.e.\ $x\in \bB_1$}\,.
    \end{equation}
    As a consequence,  $V$ is $C^{1,1-\delta}$-countably $\mathcal{H}^m$-rectifiable for any $\delta\in (0,1)$.
\end{cor}
\begin{proof}
The first part of the statement follows from Theorem \ref{decay}, leveraging on standard density arguments (notice that for $\|V\|$-a.e.\ $x$, $\bE(V,\bB_r(x))\rightarrow 0$ as $r\searrow 0$). Now, recalling \eqref{heighatdensity} of Theorem \ref{heightdue} (in scale-invariant form) we see that the second part of the statement follows by $C^{1,\alpha}$-rectifiability theorems, e.g.\ \cite[Corollary 1.6]{DelNinO}.  
\end{proof}			
We remark that it is also possible to prove the $C^{1,1-\delta}$-rectifiability  in a more direct way, relying on the Lipschitz approximation of Theorem \ref{Lipapprox}, standard arguments (with \eqref{heighatdensity} of  Theorem \ref{heightdue} in scale-invariant form), and the Whitney $C^{1,\alpha}$-extension theorem of \cite[page 177]{SteinBook}, c.f.\ the proof of \cite[Theorem 3.2]{DelellisAllard}.

\subsection{Catenoids} In this section we briefly discuss the optimality of the height bound when $m=2$ and $n=1$. It is folklore in the literature that, under these dimensional assumptions, the catenoids are the ``worst possible examples'': in this section we include the simple explicit computations for the reader's convenience. 

Let $V$ be the $2$-dimensional stationary varifold $(E, 1)$ where $E$ is a catenoid in $\RR^3$. More precisely $E$ is the surface of revolution given by the superposition of the graphs of the 	functions $\rho\mapsto\pm f(\rho)$, for $f(\rho)\defeq  \cosh^{-1}(\rho)$ and $\rho\ge 1$. 
We easily compute, in polar coordinates,
			\begin{equation}\label{dcdsssdass}
				\bE_R \defeq\bE(V, \bC_R) \defeq 2\frac{1}{\pi^2 R^2}2\pi\int_1^R \rho \bigg(2-2\frac{\sqrt{\rho^2-1}}{\rho}\bigg)\frac{\rho}{\sqrt{\rho^2-1}}\,\dd\rho\,,
			\end{equation}
			where the three factors in the integral are due to the integration in polar coordinates,  the tilt of $TV$ with respect to $\pi_0$, and the area factor, respectively. Hence,
			\begin{equation}\label{dscdsacsac}
				\frac{\bE_R}{\log(R)/R^2}\rightarrow d_1\qquad\text{as $R\rightarrow\infty$}\,,
			\end{equation}
			for a positive geometric constant $d_1$.
			Moreover, we can bound the scale-invariant height of $V$ in $\bC_R$,  for $R$ big enough (depending upon the geometric parameters),
			\begin{equation}
				2\frac{\cosh^{-1}(R)}{R}=2\frac{\log(R+\sqrt{R^2-1})}{R}=2\bigg(\frac{\log(R)}{R^2}\bigg)^{1/2}\frac{\log(R+\sqrt{R^2-1})}{\sqrt{\log(R)}}\,,
			\end{equation}
			so that 
			\begin{equation}
			\frac{2\cosh^{-1}(R)/R}{\sqrt{-\bE_R\log(\bE_R)}}\rightarrow  d_2\qquad\text{as $R\rightarrow\infty$}\,,
			\end{equation}
			for a positive geometric constant $d_2$. Hence, we see that in our case the height behaves like $\sqrt{\bE|\log(\bE)|}$.
             When $m\geq 3$, the height of the higher-dimensional catenoids can be easily shown to be comparable to $\mathbf{E}^{\frac{1}{m}}$ (the computations are straightforward and we omit them) and thus we do not know whether the logarithmic correction is really needed. 

    Concerning the Lipschitz approximation, it is well known to the experts that for stationary varifolds there is no $E^\beta$-Lipschitz approximation for which the error term in \eqref{measureestimatem} is superlinear with respect to $\bE$, as it happens for area-minimizing currents, c.f.\ \cite{DS1}.
	Indeed, assume $m=2$, $n=1$, and let $V$ be the $2$-dimensional catenoid of the example above. We wish to approximate $V$ with the graph of an $\bE_R^\beta$-Lipschitz $2$-valued function, for some $\beta\in (0,1/2)$. We can just set it equal to $\llbracket f\rrbracket + \llbracket - f\rrbracket$ outside the disk $B_{R_\beta}$, and then extend it to be constant inside the disk. Set
			\begin{equation}
				\frac{1}{\sqrt{R_\beta^2-1}}=\bE_R^\beta\,,
			\end{equation}
			and notice that, for $R$ big enough, $R_\beta<R$.
			Clearly the graph of $f$ misses the entire $\supp (V)\cap \bC_{R_\beta}$, namely the graph fails to include a region of $\supp (V)$ with surface area
			\begin{equation}
				\|V\|(\bC_{R_\beta})=2(2\pi)\int_1^{R_\beta} \rho \frac{\rho}{\sqrt{\rho^2-1}}\,\dd\rho\,,
			\end{equation}
			where again we used integration in polar coordinates.  Therefore, we can compute that the scale invariant surface of that portion of the varifold behaves like
			\begin{equation}
				\frac{\|V\|(\bC_{R_\beta})/R^2}{-\bE_R^{1-2\beta}/{\log(\bE_R)}}\rightarrow d_3\,,
			\end{equation}
			for a positive geometric constant $d_3$. Thus $f$ fails to cover a region with surface area comparable to $\frac{\bE^{1-2\beta}}{|\log(\bE)|}$. Analogous computations can be made in $m\geq 3$ for higher-dimensional catenoids.

Finally, for what concerns the issue of ``higher integrability'', we remark what follows. Concerning the $\bE_R^\beta$-Lipschitz $2$-valued approximation, we compute, for $p>2$, with analogous computations as for \eqref{dcdsssdass},
\begin{equation}
    \bigg(\frac{1}{R^2}\int_{B_R\setminus B_{R_{\beta}}} |D f|^p\bigg)^{1/p}=\bigg(2\frac{1}{\pi^2 R^2}2\pi\int_{R_\beta}^R \rho (\rho^2-1)^{-p/2}\frac{\rho}{\sqrt{\rho^2-1}}\,\dd\rho\bigg)^{1/p}\,.
\end{equation}
Then, using \eqref{dscdsacsac},  we see that 
\begin{equation}
    \bigg(\frac{1}{R^2}\int_{B_R\setminus B_{R_{\beta}}} |D f|^p\bigg)^{1/p}\frac{1}{R^{-2\frac{1+\beta(p-2)}{p}}\log(R)^{\frac{\beta(p-2)}{p}}}\rightarrow d_4\,,
\end{equation}
for some positive geometric constant $d_4$. Hence, we see that $\bigg(\frac{1}{R^2}\int_{B_R\setminus B_{R_{\beta}}} |D f|^p\bigg)^{1/p}$  grows faster than $\bE_R^{1/2}$, as $\beta<1/2$.

\subsection{Comparison with the existing literature}
The results of this paper have already been obtained in the literature: the novelty lies rather in some of the arguments, which we believe are transparent and easily readable. 

The height bound  of Theorem \ref{heightdue} has already been proved by Menne in \cite{MenneCVPDE}.  In fact Menne proved more general estimates, valid for varifolds with mean curvature bounded in $L^p$. Of course, an $L^\infty$ bound like the one of our theorem requires a sufficiently high $p$ (more precisely a $p$ high enough to ensure that $W^{2,p}$ embeds in $L^\infty$).
Moreover, Menne's statement bounds the width with $C(r,q)\bE_1^{1/q}$ in \eqref{bvefdasdue}, for a suitable constant $C(r,q)$. Our sharper statement is simply obtained by tracking the dependence of the constant $C$ on $q$ and then optimizing in $q$.

The very first Lipschitz approximation theorem with multifunctions was proved by Almgren in  \cite{Almgren00}. Almgren's approximation was revisited by Brakke \cite{Brakke}, where (a weaker statement) has been proved also for varifolds with mean curvature given by a measure. 

Finally, Brakke \cite{Brakke} proved the almost quadratic decay for the excess at almost every point for varifolds with square integrable mean curvature. Later, Menne in \cite{MenneJGA} (inspired by \cite{SchatzleJDG}) obtained the sharp bound $\limsup_{r\searrow 0}r^{-2}\bE(\B_r(x))<\infty$ for a.e.\ $x$, leading him to prove $C^2$-rectifiability. 

\subsection*{Acknowledgments}
Most of this work was carried out while C.B.\ was a PhD student at  Scuola Normale Superiore.
    This material is based upon work supported by the National Science Foundation under Grant No.\ DMS-1926686.
\section{Preliminaries}

In these preliminaries we fix a varifold $V$ and an integer $Q$ satisfying the assumptions of our statements.
We begin by recalling Allard's \emph{monotonicity formula} and \emph{isoperimetric inequality} (see \cite[Theorem 7.1]{AllardFirst}).

\begin{thm}[Monotonicity formula]
        For every $0<s<r<4$,
        \begin{equation}\label{monotformula}
        \frac{\|V\|(\bB_r)}{\omega_m r^m}- \frac{\|V\|(\bB_s)}{\omega_m s^m}=\int_{\bB_r\setminus \bB_s}\frac{| x^\perp|^2}{|x|^{m+2}}\,\dd\|V\|\ge 0\,.
        \end{equation}
    \end{thm}

    \begin{thm}[Isoperimetric inequality]
       If $\varphi\in C_c^1(\bC_4)$ is non-negative,
\begin{equation}\label{isoperimetric}
	\int_{\{\varphi\ge 1\}}\varphi\,\dd\|V\|\le C_0 \bigg(\int\varphi\,\dd\|V\|\bigg)^{1/m}\int |\nabla_{T V}\varphi|\,\dd\|V\|\,.
\end{equation}
    \end{thm}

Next we state and prove an elementary lemma, which shows that, if the excess is suitably small, then the varifold (approximately) covers the reference plane with a constant multiplicity, at definite scales.
    \begin{lem}\label{vfedacssc}
	Let $0<r_1<r_2<1$ and let $\eta\in(0,1/2)$. Assume that $\bE_1$ is small enough (depending upon $r_1,r_2,\eta$, and the geometric parameters), and that $\frac{\|V\|(\bC_1)}{\omega_m}\le Q+\frac{3}{4}$. Then, there exists $Q'\in \{0,\dots,Q\}$ such that the following holds for every $x\in \bC_{1}$ and $r >r_1$ with $\bC_r(x)\subseteq \bC_{r_2}$:
	\begin{equation}
		\frac{\|V\|(\bC_{r}(x))}{\omega_m r^m}\in (Q'-\eta,Q'+\eta)\,.
	\end{equation}
\end{lem}
\begin{proof} 
	Using the monotonicity formula we can reduce to the case in which  $\supp(V)\subseteq \RR^m\times B_{C_0}(0)$: if we set $r_3=\frac{1+r_2}{2}$ we can see that there is a finite number $N$ and a constant $C_0$ (depending on $Q$ and the geometric parameters) such that, for some suitably chosen points $y_1, \ldots y_N\in \RR^n$, not necessarily distinct,  $\supp (V)\cap \bC_{r_3}\subseteq \bigcup_i \RR^m \times B_{C_0} (y_i)$.
    
	Assume therefore by contradiction that $\supp (V) \subseteq \RR^m \times B_{C_0} (0)$ but that our claim is false. Then there exists a sequence of integral  varifolds $V_i$  as in the assumption of the lemma, with  $\bE_1(V_i,\bC_1)\rightarrow 0$, but not satisfying the conclusion of the lemma. By the Compactness Theorem of integral varifolds (e.g.\ \cite{Simon}), we have that $V_i$ converges to the integral varifold $V_\infty$ (in the sense of varifolds), up to a non-relabelled subsequence. Now, we see that  $\bE(V_\infty,\bC_1)=0$.  We argue that this implies that $V_\infty$ is supported in the union of finitely many affine planes parallel to $\pi_0$. In fact we fix a constant vector $e$ which is parallel to $\pi_0$, a function $\varphi \in C^\infty_c (\bC_1)$, and test the stationarity of $V_\infty$ with the vector field $X=\varphi e$. From the formula for the first variation we conclude that 
    \begin{equation}
            \int \partial_e \varphi\,\dd\|V_\infty\| = 0\,.
    \end{equation}
    This means that the distributional derivative of the measure $\|V_\infty\|$ in any direction parallel to $\pi_0$ is in fact constant, namely that the varifold is translation invariant along all directions parallel to $\pi_0$. Therefore $\supp (V_\infty)$ is also translation invariant. Since the latter must be a rectifiable $m$-dimensional set with locally finite measure, we conclude that it is locally the union of finitely many planes parallel to $\pi_0$. On the other hand Allard's constancy theorem implies that $\Theta (V_\infty, \,\cdot\,)$ is a constant (and a positive integer) on each of these planes. Thus, for some non-negative integer $Q_\infty\le Q$, $\frac{\|V_\infty\|(\bC_r(x))}{\omega_m r^m}=Q_\infty$, for every $\bC_r(x)\subseteq\bC_1$. This easily implies a contradiction.
	\end{proof}
    \section{Proof of the main results}
	\subsection{First height bound}
    This is our first height bound, which is weaker than Theorem \ref{heightdue}. We are going to use it to obtain a first Lipschitz approximation, which is as well weaker than Theorem \ref{Lipapproxstrong}. However this first Lipschitz approximation will allow us to improve the height bound to the optimal one of Theorem \ref{heightdue} and hence to the improved Lipschitz approximation of Theorem \ref{Lipapproxstrong}. 
	\begin{thm}[First height bound]\label{height} Let $V$ be stationary in $\bC_1$.
		Let $ r\in (0,1)$.
		Assume that $\bE_1$ is small enough (depending upon $ r$ and the geometric parameters) and that $\frac{\|V\|(\bC_1)}{\omega_m}\le Q+\frac{1}{2}$. 
  Then there are a constant $C= C(r,Q,m,n)$ and points $y_1,\dots,y_Q\in \RR^n$ (not necessarily distinct) such that 
		\begin{equation}\label{bvefdas}
			\supp(V)\cap\bC_{ r}\subseteq \bigcup_{h=1,\dots,Q} \RR^m\times\bar B_{C \bE_1^{1/(2m)}}(y_h)\,.
		\end{equation}
  Moreover, if in addition $\Theta(0)=Q$, then 
\begin{equation}
			\supp(V)\cap\bC_{ r}\subseteq  \RR^m\times \bar B_{C\bE_1^{1/(2m)}}(0)\,.
		\end{equation}
\end{thm}
The height bound of Theorem \ref{height} follows easily from the following lemma: a slice of $\bC_2$ of height comparable to $\bE_2^{1/(2m)}$ carries a definite amount of mass of the stationary varifold. Notice that a naive application of the monotonicity formula gives a lower bound for the mass comparable to $\bE_2^{1/2}$: the point is that the lower bound is as if the varifold were actually contained in a slice of height comparable to $\bE_2^{1/(2m)}$.
	\begin{lem}\label{induction}
		Assume that  $\|V\|(\bC_2)\le  S$, that $\bE_2$ is small enough, depending upon $S$ and the geometric parameters, and that  $\supp(V)\cap B_{1}(0, \pi_0)\ne \emptyset$. Then, there exists a constant $C=C(S,Q,m,n)$ such that 
		\begin{equation}\label{e:un-po-di-massa}
			\|V\|\big(\big\{(x,y)\in\RR^{n+m}: |x|<2,|y_{j}|\leq C\bE_2^{1/(2m)}\big\}\big)\ge \frac{1}{C_0}\qquad\text{for every }j=1,\dots,n\,.
		\end{equation} 
	\end{lem}
\begin{proof}  First of all, without loss of generality we can prove the claim for $\bE_2>0$. In fact, if $\bE_2=0$, we can argue as in Lemma \ref{vfedacssc} to conclude that $V$ is supported in a finite union of disks parallel to the basis of the cylinder and over each of them the density $\Theta$ is a constant positive integer. But because $\supp (V)\cap B_1 (0, \pi_0)\neq \emptyset$, one such disk must be $B_2 (0, \pi_0)$.
	We assume $j=n+m$  for definiteness and we take coordinates $(\zeta,\tau)\in \RR^{n+m-1}\times \RR$.
Define 
\begin{align}
	g(t)&\defeq \int_{\bC_2 \cap \{\tau<t\}}|\p_{T V}-\p_{\pi_0}|\,\dd\|V\|\,,\\
	f(r,s,t)&\defeq \|V\|(\{|\zeta|< r,\tau\in (s,t)\})\,.
\end{align}

Take $M\in(0,1)$  to be fixed later, in particular, we will set it as $C(S)\bE_2^{1/(2m)}$ in order to have that \eqref{tofail} necessarily fails. Notice that, being $g$ monotonic, we have  
\begin{equation}
	\int_{-2M}^{-M} g' +	\int_{M}^{2M} g' \le 
	 \int_{\bC_2 }|\p_{T V}-\p_{\pi_0}|\,\dd\|V\|\le \|V\|(\bC_2)^{1/2}\bE_2^{1/2} \le \sqrt{S} \bE^{1/2}_2\,,
\end{equation}
where we also used Holder's inequality. Therefore, we can fix $s_0\in (-2M,-M), t_0\in (M,2M)$ points of differentiability of $g$ such that
\begin{equation}
	g'(s_0)+g'(t_0)\le \frac{\sqrt{S} \bE_2^{1/2}}{M}\,,
\end{equation}
we assume that also $\|V\|(\{|\zeta|<2,\tau\in\{s_0,t_0\}\})=0$.

Now consider, for $s<t$ (which will be set as $s_0<t_0$), $r>0$, and $\epsilon\in (0,r)$,	$\varphi^\epsilon_{r,s,t}(\zeta,\tau)\defeq\phi^\epsilon_r(|\zeta|)\phi^\epsilon_{s,t}(\tau)$, where
$\phi^\epsilon_r$ and $\phi^\epsilon_{s,t}(\tau)$ are smooth and  $2/\epsilon$-Lipschitz,
\begin{align}
\phi^\epsilon_r(|\zeta|)=1\qquad&\text{for $|\zeta|<r$}\,,\\
	\phi^\epsilon_r(|\zeta|)=0\qquad&\text{for $|\zeta|>r+\epsilon$}\,,\\
	\phi^\epsilon_{s,t}(\tau)=1\qquad&\text{for $\tau \in (s+\epsilon,t-\epsilon)$}\,,\\
	\phi^\epsilon_{s,t}(\tau)=0\qquad&\text{for $\tau\ge t$ or $\tau\le s$}\,.
\end{align}
Plugging this choice into the isoperimetric inequality \eqref{isoperimetric}, we obtain
\begin{equation}
    \begin{split}
	f(r,s+\epsilon,t-\epsilon)\le C_0 f(r+\epsilon,s,t)^{1/m}\bigg(\frac{2}{\epsilon}\|V\|& (\{|\zeta|\in (r,r+\epsilon), \tau\in (s,t)\})\\&+ \frac{2}{\epsilon}\int_{\bC_{2}\cap \{\tau\in (s,s+\epsilon)\}} |\p_{TV}-\p_{\pi_0}|\,\dd\|V\|\\
	&+ \frac{2}{\epsilon}\int_{\bC_{2}\cap \{\tau\in (t-\epsilon,t)\}} |\p_{TV}-\p_{\pi_0}|\,\dd\|V\|
	\bigg)
        \end{split}
\end{equation}
as $\nabla_{TV} \phi_{s,t}^\epsilon(\tau)=(\partial_\tau \phi_{s,t}^\epsilon) \p_{TV}(e_{n+m})$ and $\p_{\pi_0}(e_{n+m})=0$.
For a.e.\ $r$  we can let $\epsilon\searrow 0$ and obtain that
\begin{equation}\label{bgvfads}
	f(r,s_0,t_0)^{(m-1)/m}\le C_0 \big( \partial_r f(r,s_0,t_0)+g'(t)+g'(s)\big)\le C_0(\partial_r f(r,s_0,t_0)+\sqrt{S}\bE^{1/2}_2/M)\,.
\end{equation}
Assume that there exists $r_0\in (3/2,2)$ such that \eqref{bgvfads} holds and $\partial_r f(r_0,s_0,t_0)\le \sqrt{S }\bE_2^{1/2}/M$. In this case,
\begin{equation}
	f(r_0,s_0,t_0)^{(m-1)/m}\le C_0\sqrt{S} \frac{\bE_2^{1/2}}{M}\,.
\end{equation}
As $s_0<-M<M<t_0$ and $M\in (0,1)$, by the monotonicity formula centered at $p\in\supp(V)\cap B_1(0)$, we can bound
\begin{equation}
	f(r_0,s_0,t_0)\ge \frac{1}{C_0} M^{m}\,,
\end{equation}
hence
\begin{equation}\label{tofail}
	M\le C_0\sqrt{S} \bE_2^{1/(2m)}\eqdef \frac{1}{2}C(S)\bE_2^{1/(2m)}\,.
\end{equation}
Now we fix  $M$ as $C(S)\bE_2^{1/(2m)}$, so that the above can not happen. We assume that $ M\in (0,1)$, which we can do, requiring that $\bE_2$ is smaller than a quantity that depends upon $S$ and the geometric parameters.

Hence, for a.e.\ $r\in (3/2,2)$, $\sqrt{S }\bE_2^{1/2}/ M<  \partial_r f(r_0,s_0,t_0) $, so that,  for a.e.\ $r\in (3/2,2)$, $f(r,s_0,t_0)^{(m-1)/m}< C_0\partial_r f(r,s_0,t_0)$, by \eqref{bgvfads}. In particular, $f(r,s_0,t_0)>0$ for every $r\in (3/2,2)$, thus
\begin{equation}
	\frac{1}{C_0}\le \frac{1}{m}\frac{\partial_r f(r,s_0,t_0)}{f(r,s_0,t_0)^{(m-1)/m}}=\partial_r \big(f(r,s_0,t_0)^{1/m}\big)\qquad\text{for a.e.\ $r\in (3/2,2)$}\,,
\end{equation}
which implies \begin{equation}\label{bvfdas}
	\|V\|(|\zeta|<2, \tau\in (-2 M,2 M)) \ge f(2,s_0,t_0)\ge \frac{1}{C_0}\, .\qedhere
\end{equation}
\end{proof}

\begin{proof}[Proof of Theorem \ref{height}]
Clearly,  there is no loss of generality in assuming that $r\in (0,1)$ is bigger than a geometric constant, which has still to be determined. Also, if $\bE_1=0$, then, arguing as in the proof of Lemma \ref{vfedacssc}, we see that $V$ in $\bC_1$ is given by the union of finitely many planes parallel to $\pi_0$, so that the conclusion trivially follows. Hence, we assume $\bE_1>0$.

	Let $\bar\rho\defeq (1+r)/2$ and set $\bar s\defeq (1-\bar\rho)/2=(1-r)/4$. 
Cover $\bC_{\bar\rho}$ with $\bC^1,\dots, \bC^M$, where $M\le C(r)$,  such that for every $i$, $\bC^i=\bC_{\bar s}(x_i)$ with  $x_i\in B_{\bar\rho}$. Notice that for every $i$, $\bC_{2\bar s}(x_i)\subseteq\bC_1$, $\bE(V,\bC_{2\bar s}(x_i))\le C(r)\bE_1$, and $\frac{\|V\|(\bC_{2\bar s}(x_i)) }{\omega_m(2\bar s)^m }\le C( r)$.
	
Fix $i=1,\dots, M$. For fixed $j=1,\dots,n$, take $p=(\bar x,\bar y)\in \supp(V)\cap \bC^i$ provided that this set is not empty. By (the scale-invariant version of) Lemma \ref{induction},   if  $\bE_1$ is small enough,
	\begin{equation}
		\|V\|\big(\big\{(x,y)\in \RR^{n+m}:|x-x_i|<2\bar s, |y_j-\bar y_j|<  C(r)\bE_1^{1/(2m)} \big\}\big)\ge \frac{1}{C(r)}\,.
	\end{equation}
	Being $\|V\|(\bC_{2\bar s}(x_i))\le C_0$, this means that there exist finitely many $\bar y_j^1,\dots,\bar y_j^{N}$ (depending on $i$), where $N\le C(r)$,  such that 
	\begin{equation}
		\supp(V)\cap\bC^{i}\subseteq \big(\big\{(x,y)\in \RR^{n+m}: |y_j-\bar y_j^h|< C(r)\bE_1^{1/(2m)}\text{ for some $h=1,\dots,N$} \big\}\big)\,.
	\end{equation}
	Repeating the same argument for every $j$-th coordinate, with $j=1,\dots,n$, we see that there are $\bar y^1,\dots,\bar y^{N'}$ (depending on $i$), where $N'\le C(r)$, such that 
		\begin{equation}\label{mark}
		\supp(V)\cap\bC^i\subseteq \big(\big\{(x,y)\in \RR^{n+m}: |y-\bar y^h|< C(r)\bE_1^{1/(2m)}\text{ for some $h=1,\dots,N'$} \big\}\big)\,.
	\end{equation}
	Hence, applying this argument to every $\bC^i$, we see that there exist $\hat y^1,\dots,\hat y^{N''}$, where $N''\le C(r)$, such that 
		\begin{equation}
	\supp(V)\cap\bC_{\bar\rho}\subseteq \big(\big\{(x,y)\in \RR^{n+m}: |y-\hat y^h|< C(r)\bE_1^{1/(2m)}\text{ for some $h=1,\dots,N''$} \big\}\big)\,.
\end{equation}
	
Hence, there exist $S_1,\dots,S_{K}\subseteq\RR^n$ open and pairwise disjoint, where $K\le C(r)$, with  $\diam(S_h)\le C(r)\bE_1^{1/(2m)}$ and 
\begin{equation}
		\supp(V)\cap\bC_{\bar\rho}\subseteq \bigcup_h \RR^m\times S_h\,.
\end{equation}
	Set then, for every $h=1,\dots,K$, $V_h\defeq V\mres (\RR^m\times S_h)$ and notice that $V_h$ is stationary in $\bC_{\bar\rho}$. 
	Now take any $h=1,\dots,K$. Assume that $\exists p\in \supp(V_h)\cap\bC_{ r}$. 
	We apply Lemma \ref{vfedacssc}  (in scale-invariant form) with 
 \begin{equation}
     0<r_1=(\bar\rho- r)/4=(1-r)/8<r_2=(3\bar\rho+r)/4=(5r+3)/8<\bar\rho=(1+r)/2\quad\text{and}\quad\eta=1-\bar\rho=(1-r)/2\,,
 \end{equation} provided that $\bE_1$ is small enough, to see that 
	\begin{equation}
	    	\|V_h\|(\bC_{r_2})\ge (Q'-\eta )\omega_mr_2^m\quad\text{and}\quad 	\|V_h\|(\bC_{r_1})\le (Q'+\eta )\omega_mr_1^m\,,
	\end{equation}
	for some $Q'\in\NN$ that depends on $h$. By the monotonicity formula centered at $p$, we see that $\|V_h\|(\bC_{r_1}(p))\ge\|V_h\|(\bB_{r_1}(p))\ge \omega_m r_1^m$, so that  $1\le (Q'+(1-r)/2)$, hence, as $r>0$, $Q'\ge 1$. Therefore,
	\begin{equation}\label{btesfdsdvv}
		 \|V_h\|(\bC_{r_2})\ge ((1+r)/2)\omega_m ((5r+3)/8)^m\,.
	\end{equation}
	
	Summing \eqref{btesfdsdvv} over all $h$ such that $\supp(V_h)\cap\bC_{ r}\ne \emptyset$, we obtain that 
	\begin{equation}
	    	|\{h:\supp(V_h)\cap\bC_{ r}\ne \emptyset\}|((1+r)/2)\omega_m ((5r+3)/8)^m\le \|V\|(\bC_{r_2})\le \|V\|(\bC_1)\le \omega_m (Q+1/2)\,.
	\end{equation}
	Hence, if $r\in (0,1)$ is bigger than a geometric quantity, we have that $	|\{h:\supp(V_h)\cap\bC_{ r}\ne \emptyset\}|\le Q$, thus concluding the proof of \eqref{bvefdas}, as $\supp(V_h)\subseteq\RR^m\times S_h$, where $\diam(S_h)\le C( r)\bE_1^{1/(2m)}$.

    Now, assume that $\Theta(0)=Q$. As $0\in\supp(V)$, $0\in \supp(V_h)$ for some $h$, up to reordering assume that $0\in \supp(V_1)$.  By the monotonicity formula,
    \begin{equation}\label{bvrfdvsasss}
         \|V_1\|(\bC_{r_2})\ge \|V_1\|(\bB_{r_2})\ge Q\omega_m ((5r+3)/8)^m\,. 
    \end{equation}
    We  conclude the proof by showing that $\supp(V_h)\cap \bC_r=\emptyset$ for every $h=2,\dots,K$. By contradiction, take $h=2,\dots,K$ not satisfying the above. Summing \eqref{bvrfdvsasss} and \eqref{btesfdsdvv}, we obtain
    \begin{equation}
        Q\omega_m ((5r+3)/8)^m+((1+r)/2)\omega_m ((5r+3)/8)^m\le \|V\|(\bC_{r_2})\le \|V\|(\bC_1)\le \omega_m (Q+1/2)\,,
    \end{equation}
    which is a contradiction,  if $r\in (0,1)$ is bigger than a geometric quantity. 
\end{proof}

\subsection{First Lipschitz approximation}
Now we state and prove a slightly weaker form of Theorem \ref{Lipapproxstrong}, which is an easy consequence of the height bound of Theorem \ref{height} and a classical ``maximal function truncation'' argument.
\begin{thm}[First Lipschitz approximation]\label{Lipapprox} Let $V$ be stationary in $\bC_4$.
	Assume that $\bE_4$ is small enough (depending upon the geometric parameters), that $\frac{\|V\|(\bC_4)}{\omega_m 4^m}\le Q+1$, and that
	\begin{equation}\label{vfeadsc}
 \frac{\|V\|(\bC_{3})}{\omega_m 3^m}  \in(Q-\textstyle{\frac{1}{2}},Q+\textstyle{\frac{1}{2}})\,.
	\end{equation} 
	Set, for $\lambda \in (0,1)$ small enough (depending upon the geometric parameters),
	\begin{equation}\label{e:def-K}
		K_\lambda\defeq \{x\in B_1:\bmax\be (x)\le \lambda\}\,.
	\end{equation}
Then, there exists a $Q$-valued Lipschitz function $f:B_1\rightarrow\Iqs$ such that 
	\begin{equation}\label{Lipdcsa}
		f\text{ is } C_0\lambda^{1/(2m)}\text{-Lipschitz}\, .
	\end{equation}
In addition (i), (ii), (iv), and (v) in Theorem \ref{Lipapproxstrong} hold, while in place (iii) we claim the estimate 
\begin{equation}\label{measureestimate}
		|B_1\setminus K_\lambda|+\|V\|(\bC_1\setminus (K_\lambda\times\RR^n))\le C_0^q\frac{1}{\lambda^{q/2}}\int_{\bC_4}|\p_{T V}-\p_{\pi}|^{q}\,\dd\|V\|\qquad \text{for every } 2\leq q<\infty\, .
\end{equation}
\end{thm}

\begin{proof}
 First of all we can assume that $\bE_4>0$, otherwise we can argue as in Lemma \ref{vfedacssc} to conclude that $\supp (V)\cap \bC_3$ consists of a finite number of disks parallel to the basis of $\bC_3$ and over each of them the density of $V$ is a constant positive integer. In particular the varifold is itself the graph of a constant multivalued function.
    
	We start the proof with a preliminary observation, which we state for any stationary varifold, as we are going to apply it to $V$ and to some parts of it. 
	\\\textbf{Step 1.} Let $W$ be a stationary varifold in $\bC_{4}$, with $\frac{\|W\|(\bC_{{7/2}})}{\omega_m (7/2)^m}\le Q+\frac{1}{2}$. Assume that $\bC_{\rho}(x)$ intersects $\bC_1$ and that 
	\begin{equation}\label{vfeadsac}
		\bE(W,\bC_s(y))\le \mu \qquad\text{for every $\bC_{\rho}(x)\subseteq\bC_s(y)\subseteq\bC_{4}$}\,,
	\end{equation}
	for some $\mu\in (0,1)$.
	Then, if $\mu$ is small enough (depending upon the geometric parameters), which we will now on assume, there exists $Q'\in\{0,\dots,Q\}$ such that
	\begin{equation}\label{brvfd}
		\frac{\|W\|(\bC_s(y))}{\omega_m s^m}\in (Q'-1/(2Q),Q'+1/(2Q))\qquad\text{for every $\bC_\rho(x)\subseteq\bC_s(y)\subseteq\bC_{7/2}$}\,.
	\end{equation}
	Notice that, in particular, if $x\in B_1$ is such that \eqref{vfeadsac} holds for every $x\in \bC_s(y)\subseteq\bC_{4}$, then \eqref{brvfd} holds for every $x\in \bC_s(y)\subseteq\bC_{7/2}$.
	
	Indeed, take $\mu$ small enough (depending upon the geometric parameters) so that Lemma \ref{vfedacssc} applies with parameters $r_1=1/64$, $r_2=15/16$, and $\eta=1/(2Q)$, yielding $Q'\in\{0,\dots,Q\}$. Take any $\bC_\rho(x)\subseteq \bC_s(y)\subseteq \bC_{7/2}$ as above. If $s\ge 1/16$, then \eqref{brvfd} follows by the choice of $r_1$, indeed, we apply the scale-invariant version of Lemma \ref{vfedacssc}, for the radii $1/16<15/4<4$.  In the case $s<1/16$, we argue as follows. For $l\ge 0$, consider the cylinders $\bC^l(y)\defeq \bC_{1/2^{l+3}}(y)\subseteq\bC_{15/4}$. We claim that \eqref{brvfd} holds for $s=1/2^{l+3}$, for every $l\ge 1$ with $\bC_\rho(x)\subseteq\bC^l(y)$. The base case is observed above. For the inductive step at $l+1$,  we use (the scale-invariant form of) Lemma \ref{vfedacssc} on $ \bC^{l-1}(y)$, with the inductive assumption.
	For general $s$, there exists a unique $l'\in\NN$, $l'\ge 1$, with $s\in [1/2^{l'+4},1/2^{l'+3})$, then we use again Lemma \ref{vfedacssc} on $ \bC^{l'-1}(y)$ together with the claim for $s=1/2^{l'+3}$.
		\medskip\\\textbf{Step 2.}	Now we ask that  $\bE_4$ is small enough (depending upon the geometric parameters) so that $ \frac{\|V\|(\bC_{7/2})}{\omega_m(7/2)^m}\in (Q-1/2,Q+1/2)$ (using	Lemma \ref{vfedacssc} with the assumption \eqref{vfeadsc}). Moreover, we ask that $\lambda\le \mu$. Now, \textbf{Step 1} applies, notice that $Q'$ as in \eqref{brvfd} for $V$ is exactly $Q$, by \eqref{vfeadsc}. In particular, we have proved \eqref{vadfscas}. Actually, we have proved the claim for a slightly larger family of cylinders.
	\medskip\\\textbf{Step 3.}
	Now take $x\in K_\lambda$ and $B_{s}(y)\subseteq B_3$ with $x\in B_{s}(y)$ and consider $\bC_{7s/6}(y)\subseteq\bC_{7/2}$. We can assume that $\lambda$ is small enough (depending upon the geometric parameters) to apply   Theorem \ref{height} with $\bar r=6/7$.
	Hence, by  (the scale-invariant version of) Theorem \ref{height} applied to $\bC_{7s/6}(y)$ (recalling \textbf{Step 1} for the bound on the measure), there exist $S_1^{y,s},\dots,S^{y,s}_{N^{y,s}}$ open and pairwise disjoint, with $\diam(S_h^{y,s})\le D\lambda^{1/(2m)}s$ (with $D\le C_0$, but independent of $y$ and $s$) and $N^{y,s}\le Q$, such that 
	\begin{equation}\label{vbgrfsdc}
	\supp(V)\cap\bC_{s}(y) \subseteq\bigcup_{h=1,\dots,N^{y,s}} \RR^m\times S_h^{y,s}\,.
	\end{equation}
	We also assume that for every $h$, $	\supp(V)\cap\bC_{s}(y) \cap(\RR^m\times S_h^{y,s})\ne \emptyset$.
	Set then $V_h^{y,s}\defeq V\mres (B_{s}(y)\times S_{h}^{y,s})$, notice that $V_h^{y,s}$ is a stationary varifold in $\bC_s(y)$.
	Moreover, again by (the scale-invariant version of) \textbf{Step 1} (for both $V_h$ and $V$), we obtain that there exists $Q^{y,s}_h\in\{0,\dots,Q\}$ with   $\sum_h Q_h^{y,s}=Q$ and
	\begin{equation}\label{vfsedac}
		\frac{\|V_h^{y,s}\|(\bC_\rho(y'))}{\omega_m \rho^m}\in (Q_h^{y,s}-1/(2Q),Q_h^{y,s}+1/(2Q))\qquad\text{for every $\bC_\rho(y')\subseteq\bC_{7s/8}(y) $ intersecting $K_\lambda$}\,.
	\end{equation}
	
	Now fix $x\in K_\lambda$ and let $s\in (0,1)$. Denote by $S_1^{x,s},\dots,S_{N^{x,s}}^{x,s}$ the sets as for \eqref{vbgrfsdc} for $B_{s}(x)$, and choose points  $p_1^{x,s},\dots,p_{N^{x,s}}^{x,s}\in\RR^n$  with $p_h^{x,s}\in S^{x,s}_h$ and such that $\supp(V)\cap (B_{s}(x)\times\{p_h^{x,s}\})\ne\emptyset$. 
	Now notice that thanks to \eqref{vfsedac}, 
	\begin{equation}
		\dist_H\big(\{p_1^{x,s},\dots,p_{N^{x,s}}^{x,s}\},\{p_1^{x,\rho},\dots,p_{N^{x,\rho}}^{x,\rho}\}\big)\le 2D\lambda^{1/(2m)} s\qquad\text{for every }\rho\in (0,7s/8]\,,
	\end{equation}
	so that we have a limit set $\{p^x_1,\dots,p^x_{N^x}\}$, where $N^x\le Q$, and $(x,p_h^x)\in\supp(V)$ for every $h$.
	
	Now take $\bar s\in (0,1)$ small enough so that $B_{4D\lambda^{1/(2m)} \bar s}(p^x_h)$ are pairwise disjoint.  Take $s\in (0,\bar s)$, consider $S_1^{x,s},\dots,S_{N^{x,s}}^{x,s}$. Take $h$ and consider $V^{x,s}_h$, notice that there exists a set $T_h^{x,s}$ with $\supp(V_h)\subseteq  B_{s}(x)\times T_h^{x,s}$ and $\overline{T_h^{x,s}}\subseteq S_h^{x,s}$. By \eqref{vfsedac}, and exploiting $T_h^{x,s}$, we find a point  $(x,\bar p)\subseteq(\{x\}\times \overline{T_h^{x,s}})\cap\supp(V^{x,s}_h)\subseteq\supp(V)\cap S_h^{x,s}$, and this forces $\bar p=p^x_{h'}$ for some $h'$. We have thus seen that for every $h$, there exists a unique $h'$ (as $|p^x_{h'}-p^x_{h''}|\ge 8D\lambda^{1/(2m)}\bar s>\diam(S^{x,s}_h)$ for $h''\ne h'$) such that $p^x_{h'}\in S^{x,s}_h$. Conversely, for every $k$, there exists a unique $k'$ such that $p^x_{k}\in S^{x,s}_{k'}$, as $p^x_{k}\in\supp(V)$. Hence we have a one-to-one correspondence between the points  $\{p^x_1,\dots,p^x_{N^x}\}$ and the sets $S_1^{x,s},\dots,S_{N^{x,s}}^{x,s}$. Hence, to ${p^x_h}$, we associate $Q_{h'}^{x,s}$ (for ${p^x_h}\in S^{x,s}_{h'}$) as the relevant integer as in \eqref{vfsedac}. Notice that \eqref{vfsedac} implies that this choice is well posed, i.e.\ independent of $s$, and recall that $\sum_h Q_h^{x,s}=Q.$ Therefore, we can define $f(x)\defeq \sum_h Q^{x,s}_h\a{p_h^{x}}\in\Iqs$.  Notice also that the above argument implies $\supp(V)\cap(\{x\}\times\RR^n )= \bigcup_h (x,p_h^x)$. In particular,  \eqref{brvfeads}  follows from \eqref{vbgrfsdc}.
			\medskip\\\textbf{Step 4.}
			We prove that for every $x_1,x_2\in K_\lambda$, then $\cG(f(x_1),f(x_2))\le C_0\lambda^{1/(2m)}|x_1-x_2|$. This will prove	\eqref{Lipdcsa} for $f: K_\lambda\rightarrow \Iqs $, and hence, thanks to \cite[Theorem 1.7]{DSq}, will provide us with the sought $f$ satisfying 	\eqref{Lipdcsa}. Set $\bar r\defeq |x_1-x_2|$ and $\bar x\defeq (x_1+x_2)/2$. Clearly, $x_1,x_2\in B_{ 7\bar r/8}(\bar x)\subseteq B_{\bar r}(\bar x)\subseteq B_3$. 
			We use now \textbf{Step 3}, with the same notation. Consider  $p_h^{x_1}\subseteq S^{\bar x,\bar r}_{h'}$ for a unique $h'$. Therefore, if $s$ is small enough, $S_h^{x_1,s}\subseteq S_{h'}^{\bar x,\bar r}$.  We assume that $s$ is small enough so that this holds for every $h$, and also that $B_{4D\lambda^{1/(2m)}s}(p_h^{x_1})$ are pairwise disjoint.  Therefore, combining \eqref{vfsedac} both for $V^{x_1,s}_h$ and $V^{\bar x,\bar r}_{h'}$, we see that
			\begin{equation}
				\sum_{h:S_h^{x_1,s}\subseteq S_{h'}^{\bar x,\bar r}} Q_h^{x,s}=Q_{h'}^{\bar x,\bar r},
			\end{equation}
		or, alternatively, that \begin{equation}
							\sum_{h:p_h^{x_1}\in  S_{h'}^{\bar x,\bar r}} Q_h^{x_1}=Q_{h'}^{\bar x,\bar r}\,.
			\end{equation}
			The same consideration holds for $x_2$. Recalling that $\diam(S_{h'}^{\bar x,\bar r})\le D\lambda^{1/(2m)}\bar r$, we immediately obtain the sought bound.
		\medskip\\\textbf{Step 5.} We prove \eqref{measureestimate}.
		Take $y\in B_1\setminus K_\lambda$. Take $y\in \bC_{r'}(y')\subseteq\bC_4$ such that $\bE(V,\bC_{r'}(y'))>\lambda$ and
		\begin{equation}
			r'>1/2 \sup\{r'': y\in \bC_{r''}(y'')\subseteq\bC_4:\bE(V,\bC_{r''}(y''))> \lambda\}\,.
		\end{equation}
		If $\bC_{5r'}(y')\not\subseteq\bC_{7/2}$, then $r'\ge 5/12$. Hence, by
		\begin{equation}
			\lambda<\bE(V,\bC_{r'}(y'))\le \frac{\omega_m 4^m \bE_4}{\omega_m (r')^m}\,,
		\end{equation}
		we infer that $\lambda/\bE_4\le C_0$, hence \eqref{measureestimate} is trivial. Therefore we assume that for every $y\in B_1\setminus K_\lambda$, $y\in\bC_{5r'}(y')\subseteq\bC_4$, where we used the notation as above. This means that for every $\bC_{t}(z)$  (and there exists at least one such cylinder) with $\bC_{5r'}(y')\subseteq \bC_{t}(z)\subseteq \bC_4$, then $\bE(V, \bC_{t}(z))\le \lambda$. By \textbf{Step 1}, we have that 
		\begin{equation}
			\frac{\|V\|(\bC_{5r'}(y'))}{\omega_m (5r')^m}\le Q+1/2\,.
		\end{equation}
		In particular, for every $y\in B_1\setminus K_\lambda$, we can find $y\in \bC_{r'}(y')$ with $\bE(V,\bC_{r'}(y'))>\lambda$ and $\|V\|(\bC_{5r'}(y'))\le C_0 (r')^m$, so that, by Holder's inequality,
  \begin{equation}\label{bgfsdvcs}
      \lambda < \frac{1}{\omega_m (r')^m}\int_{\bC_{r'}(y')} |\p_{T V}-\p_{\pi_0}|^2\,\dd\|V\|\le \frac{1}{\omega_m (r')^m} \bigg(\int_{\bC_{r'}(y')} |\p_{T V}-\p_{\pi_0}|^{q}\,\dd\|V\|\bigg)^{2/q} (C_0 (r')^m)^{1-2/q}\,,
  \end{equation}
  which implies 
  \begin{equation}\label{bgfsdvcsa}
      (r')^m\le C_0^{q/2}\frac{1}{\lambda^{q/2}}\int_{\bC_{r'}(y')}|\p_{T V}-\p_{\pi_0}|^{q}\,\dd\|V\|\,.
    \end{equation}
  Hence, \eqref{measureestimate} follows from a standard covering argument.
		\medskip\\\textbf{Step 6.} We prove  \eqref{densdvacs}.  Up to removing from $K_\lambda$ a negligible subset, we can assume that for every $z\in (K_\lambda\times \RR^n)\cap \supp(V)$, $\Theta(z)\in\N$. Fix $z=(x,y)$ as in the statement.
        We are going to use the same notation as in \textbf{Step 3}. Take $s$ so small that $B_{4D\lambda^{1/(2m)} s}(p_h^x)$ are pairwise disjoint, fix  $h$ such that $p_h^x=y$. 
If $s'\in (0,7s/8)$, we have trivially
\begin{equation}
	\frac{\|V_h^{x,s}\|(\B_{s'}(z))}{\omega_m {(s')}^m} \le \frac{\|V_h^{x,s}\|(\bC_{s'}(x))}{\omega_m {(s')}^m}\,.
\end{equation} On the other hand, let $s'\in (0,7s/8)$ be so small so that $\sqrt{1+D^2\lambda^{1/m}}s'<7s/8$. We have by \textbf{Step 3} that $\supp(V_h^{x,s})\cap \bC_{s'}(x)\subseteq \RR^m\times S_{h'}^{x,s'}$, with $\diam(S^{x,s'}_{h'})\le D\lambda^{1/(2m)}s'$, hence
\begin{equation}
	\frac{\|V_h^{x,s}\|(\bC_{s'}(x))}{\omega_m {(s')}^m}\le  \frac{\|V_h^{x,s}\|\Big(\B_{\sqrt{1+D^2\lambda^{1/m}}s'}(z)\Big)}{\omega_m {(s')}^m}=\frac{\|V_h^{x,s}\|\Big(\B_{\sqrt{1+D^2\lambda^{1/m}}s'}(z)\Big)}{\omega_m {(\sqrt{1+D^2\lambda^{1/m}}s')}^m} (1+D^2\lambda^{1/m})^{m/2}\,.
\end{equation}
We can let now $s'\searrow 0$, recalling \eqref{vfsedac}, to infer that
\begin{equation}
\Theta(z)\le Q^{x}_h+1/(2Q)\quad\text{and}\quad	Q_h^x-1/(2Q)\le \Theta(z)(1+D^2\lambda^{1/m})^{m/2}\,.
\end{equation}
Recalling that  $\Theta(z)$ is an integer, we see that $\Theta(z)=Q_h^{x}$, if $\lambda$ is small enough, depending upon the geometric parameters.

Finally, \eqref{vefadcsxac} follows from the construction of $f$ in \textbf{Step 3} and \cite[Equation (1.8)]{DSq}.
\end{proof}

For future reference, we record a couple of useful properties of the Lipschitz approximation.
\begin{prop}\label{lipen} Let $V$ be stationary in $\bC_4$.
    Assume that $\bE_4$ is small enough (depending upon the geometric parameters), that $\frac{\|V\|(\bC_4)}{\omega_m 4^m}\le Q+1$, and that
	\begin{equation}
 \frac{\|V\|(\bC_{3})}{\omega_m 3^m}  \in(Q-1/2,Q+1/2).
	\end{equation} 
	Consider, for $\lambda\in (0,1)$ small enough, depending upon the geometric parameters,	the  $Q$-valued Lipschitz function $f:B_1\rightarrow\Iqs$ given by Theorem \ref{Lipapprox}.
 Then, there is a geometric constant $C_0$ such that the following estimates hold
 \begin{equation}\label{lipen1}
     \int_{B_1}|D f|^{q}\le C_0^{q}\frac{1}{\lambda^{q/2}}\int_{\bC_4} |\p_{\pi_0}-\p_{T_z V}|^{q}\,\dd\|V\| \qquad \text{for every } 2\leq q<\infty\, ,
 \end{equation}
 \begin{equation}\label{lipen2}
    \bigg| \int_{B_1}\sum_{i=1}^Q\nabla \varphi\,\cdot\,D f_i\bigg|\le C_0 \|\nabla\varphi\|_{L^\infty}\frac{\bE_4}{\lambda^{-1}} \qquad\text{for every $\varphi\in C_c^1(B_1)$}\, ,
 \end{equation}
 \begin{equation}\label{area}
     \Big|\|V\|(B\times \RR^n)- Q|B|\Big|\le C_0\int_{B\times\RR^n} |\p_{T V}-\p_{\pi}|^2\,\dd\|V\|\qquad\text{for any $B\subseteq K_\lambda$ Borel}
 \end{equation}
 ($K_\lambda$ is the set defined in \eqref{e:def-K}).
\end{prop}
\begin{proof}
We set $\Gamma_f$ to be the integral varifold given by the graph of $f$, with a slight abuse we will denote with $\Gamma_f$ both the graph of $f$ and the induced varifold. In order to define such object, one can either use the theory of \cite{DSq}, namely \cite[Definition 1.10]{DSq} to obtain the current $\bG_f$ and then look at the varifold naturally induced by such current, or use \cite[Lemma 1.1]{DSq} and build the varifold $\Gamma_f$ directly. We follow here the second approach. As we are going to use the claim of \cite[Lemma 1.1]{DSq} repeatedly, we recall its statement, in our context. Namely,  we have a sequence of pairwise disjoint Borel subsets of $B_1$, $(C_h)_h$, with $\mathcal{L}^m(B_1\setminus \bigcup_h C_h)=0$, such that 
\begin{enumerate}
	\item for every $h$, $f_{|C_h}=\sum_{i=1}^Q \a{f^h_i}$, with $f^h_i$ $C_0\lambda^{1/(2m)}$-Lipschitz on $C_h$ for every $i=1,\dots,Q$,
	\item for every $h$ and $i,i'\in \{1,\dots,Q\}$, either $f^h_i=f^h_{i'}$ on $C_h$, or $f^h_i\ne f^h_{i'}$ for every $x\in C_h$,
	\item for every $h$, $D f(x)=\sum_{i=1}^Q \a{D f^h_i}$ for a.e.\ $x\in C_h$.
\end{enumerate}

Now notice that for $\mathcal{H}^m$-a.e.\ $(x,y)\in\supp(\Gamma_f)$, if we write $f(x)=\sum_i \a{p_i}$, then
\begin{equation}\label{vfdsacxa}
	\Theta(\Gamma_f,(x,y)) =\sum_{i: p_i=y} Q_i\,.
\end{equation}
This follows from \cite[Proposition 1.4]{DSsns} together with what remarked just after \cite[Definition 1.10]{DSsns}.
In particular, recalling also \eqref{densdvacs},
\begin{equation}\label{fadc}
	\Theta(\Gamma_f,z) =\Theta(z) \qquad\text{for $\mathcal{H}^m$-a.e.\ $z\in (K_\lambda\times \RR^n)\cap\supp(V)$}\,,
\end{equation}
and we recall also that, by differentiation of measures,
\begin{equation}\label{fadc1}
	T_z\Gamma_f=T_z V \qquad\text{for $\mathcal{H}^m$-a.e.\ $z\in (K_\lambda\times \RR^n)\cap\supp(V)$}\,.
\end{equation}

To fix the notation, we take the standard basis of $\{0\}\times\RR^{n}\subseteq\RR^m\times\RR^n$ to be $(e_1,\dots,e_n)$.
Take now $j=1,\dots,n$ and  $\varphi\in C_c^1(B_1)$. With an harmless abuse, we test the variation of $ V$ with the vector field $X(x,y)=\varphi(x) e_j$, so $\delta V(X)=0$, being $ V$ stationary. Therefore,
\begin{equation}\label{cdsaasc2}
\begin{split}
	&\bigg|	\int \nabla_{T_z\Gamma_f}\varphi(z)\,\cdot\, e_j\,\dd\|\Gamma_f\|(z)\bigg|
\\
&\qquad\le \bigg|\int_{\Gamma_f} \nabla_{T_z\Gamma_f}\varphi(z)\,\cdot\, e_j\Theta(\Gamma_f,z)\,\dd\mathcal{H}^m(z)-\int_{\supp(V)} \nabla_{T_z V}\varphi(z)\,\cdot\, e_j\Theta(z)\,\dd\mathcal{H}^m(z)\bigg|
\\
&\qquad \le \|\nabla \varphi\|_{L^\infty} \big(\|V\|(\bC_1\setminus (K_\lambda\times\RR^n))+Q\mathcal{H}^m\mres\Gamma_f( \bC_1\setminus (K_\lambda\times\RR^n) )\big)\le C_0\|\nabla \varphi\|_{L^\infty} \lambda^{-1}\bE_4\,,
\end{split}
\end{equation}
where we used  \eqref{fadc} and  \eqref{fadc1} to deal with the portion on $K_\lambda\times \RR^n$, and the next-to-last inequality is by \eqref{measureestimate}.

Now, we concentrate on a set $C_h$, and fix $\hat f=f^h_i$ for some $i\in\{1,\dots,Q\}$, $\hat f_j= \hat f\,\cdot\,e_j$ will denote the $j$-th component of $\hat f$.   We will implicitly extend $\hat f$ to be defined on $B_1$, and still $C_0\lambda^{1/(2m)}$-Lipschitz.
By the area formula,
\begin{equation}
	\int_{C_h\times \RR^n}\nabla_{T_z\Gamma_{\hat f}} \varphi(z)\,\cdot\,e_j\,\dd\mathcal{H}^m\mres \Gamma_{\hat f}(z)=\int_{C_h}\nabla_{T_{\hat f(x)}\Gamma_{\hat f}} \varphi(x)\,\cdot\,e_j J\hat f(x)\,\dd \mathcal{L}^m(x)\,,
\end{equation}
and we can follow the proof of  \cite[Theorem 3.4]{DelellisAllard} to obtain the following relations:
\begin{equation}
	\big|	\nabla_{T_z \Gamma_{\hat f}} \varphi\,\cdot\, e_j-\nabla \varphi\,\cdot\,\nabla \hat f_j\big|\le C_0 |\nabla\varphi(x)||\nabla  \hat f(x)|^2\qquad\text{for $\mathcal{H}^m$-a.e.\ $z=(x,y)\in \Gamma_{\hat f}$}\,,
\end{equation}
\begin{equation}\label{areafactor}
	|J\hat f(x)-1|\le C_0|\nabla \hat f(x)|^2\qquad\text{for a.e.\ $x\in B_1$}\,,
\end{equation}
and 
\begin{equation}\label{vefcdsxa}
	2|\p_{\pi_0}-\p_{T_z\Gamma_{\hat f}}|^2\ge |\nabla \hat f(x)|^2\qquad\text{for $\mathcal{H}^m$-a.e.\ $z=(x,y)\in \Gamma_{\hat f}$}\,,
\end{equation}
and these hold provided that $\lambda$ is smaller a constant depending upon the geometric parameters (thanks to \eqref{Lipdcsa}), we will fix $\lambda$ accordingly.
We thus deduce that
\begin{equation}
	\bigg| \int_{C_h}\nabla \varphi\,\cdot\,\nabla \hat f_j - \int_{C_h\times \RR^n}\nabla_{T_z\Gamma_{\hat f}} \varphi\,\cdot\,e_j\,\dd\mathcal{H}^m(z)\bigg| \le C_0\|\nabla\varphi\|_{L^\infty}\int_{C_h}|\nabla \hat f|^2\,.
\end{equation}
Now we recall that $\hat f= f_i^h$, we sum the previous inequality for $i=1,\dots,Q$ and then on $h$,  to obtain that (the scalar product with $e_j$ is understood in the following sense: $\big(\sum_i \a{P_i}\big)\,\cdot\,e_j=\sum_i 	\a{P_i\,\cdot\,e_j}$)
\begin{equation}\label{cdsaasc}
\begin{split}
			&\bigg| \int_{B_1}\sum_{i=1}^Q\nabla \varphi\,\cdot\,D (f_i\,\cdot\,e_j)- \int\nabla_{T_z\Gamma_{f}} \varphi\,\cdot\,e_j\,\dd\|\Gamma_f\|(z)\bigg|
\end{split}
		\le C_0 \|\nabla\varphi\|_{L^\infty}\sum_h\sum_{i=1}^Q\int_{C_h}|\nabla  f^h_i|^2\,,
\end{equation}
where we took into account  \eqref{vfdsacxa}.
We compute further, by \eqref{vefcdsxa},
\begin{equation}\label{cdsaasc1}
\begin{split}
	\frac{1}{2^{q/2}}\sum_h\sum_{i=1}^Q\int_{C_h}|\nabla  f^h_i|^{q}
&\le  
\sum_h\sum_{i=1}^Q\int_{C_h\times \RR^n}|\p_{\pi_0}-\p_{T_z\Gamma_{ f_i^h}}|^{q}\,\dd\mathcal{H}^m\mres \Gamma_{ f_i^h}(z)
\\
&=
\int_{\bC_1\cap \Gamma_f}|\p_{\pi_0}-\p_{T_z\Gamma_{ f}}|^{q}\Theta(\Gamma_f,z)\,\dd\mathcal{H}^m(z)
		\\
	&= \int_{ \Gamma_f\cap (K_\lambda\times \RR^n)}|\p_{\pi_0}-\p_{T_z\Gamma_{ f}}|^{q}\Theta(\Gamma_f,z)\,\dd\mathcal{H}^m(z)
	\\
	&\qquad+  \int_{ \Gamma_f\cap ((B_1\setminus K_\lambda)\times \RR^n)}|\p_{\pi_0}-\p_{T_z\Gamma_{ f}}|^{q}\Theta(\Gamma_f,z)\,\dd\mathcal{H}^m(z)
	\\
        &\le \int_{\bC_1} |\p_{\pi_0}-\p_{T_z V}|^{q}\,\dd\|V\|+C_0^q\frac{1}{\lambda^{q/2}}\int_{\bC_4} |\p_{\pi_0}-\p_{T_z V}|^{q}\,\dd\|V\|
        \\
        &\le C_0^q\frac{1}{\lambda^{q/2}}\int_{\bC_4} |\p_{\pi_0}-\p_{T_z V}|^{q}\,\dd\|V\|
\end{split}
\end{equation}
where we used \eqref{vfdsacxa}, \eqref{fadc}, \eqref{fadc1} and \eqref{measureestimate} for the next-to-last inequality. This proves \eqref{lipen1}.

Finally, collecting the outcomes of \eqref{cdsaasc}, \eqref{cdsaasc1} (for $q=2$), and \eqref{cdsaasc2},
we arrive at 
\begin{equation}\label{toharm1}
		\bigg| \int_{B_1}\sum_{i=1}^Q\nabla \varphi\,\cdot\,D (f_i\,\cdot\,e_j)\bigg|\le C_0 \|\nabla\varphi\|_{L^\infty}\lambda^{-1}\bE_4 \qquad\text{for every $\varphi\in C_c^1(B_1)$}\,,
\end{equation}
which is \eqref{lipen2}.

Now we notice that  \eqref{area} follows from  the area formula, \eqref{areafactor} with \eqref{vefcdsxa}, and \eqref{fadc} with  \eqref{fadc1}, by the argument we have used above. Alternatively, it immediately follows by the interpretation of $\Gamma_f$ as a current.
\end{proof}

\subsection{Improved height bound}
With the Lipschitz approximation Theorem \ref{Lipapprox}, we can improve the height bound of Theorem \ref{height}.
\begin{proof}[Proof of \eqref{bvefdasdue} of  Theorem \ref{heightdue}] 
We begin the  proof as for Theorem \ref{height}, in particular, we assume $\bE_1>0$. 	Let $\bar\rho\defeq (1+r)/2$ and set $\bar s\defeq (1-\bar\rho)/8=(1-r)/8$. 
Cover $\bC_{\bar\rho}$ with $\bC^1,\dots, \bC^M$, where $M\le C(r)$,  such that for every $i$, $\bC^i=\bC_{\bar s}(x_i)$ with  $x_i\in B_{\bar\rho}$. Notice that for every $i$, $\bC_{4\bar s}(x_i)\subseteq\bC_1$, $\bE(V,\bC_{4\bar s}(x_i))\le C(r)\bE_1$.
Assume that $\bE_1$ is small enough, so that we can apply Lemma \ref{vfedacssc} with 
\begin{equation}
    0<r_1=\bar s< r_2=1-\bar s=(7+r)/8<1\quad\text{and}\quad\eta=1/4\,,
\end{equation}
for some $Q'\in\NN$, $Q'\in \{0,\dots,Q\}$.

We are going to  show that for every $i$,
\begin{equation}\label{asdcas}
\supp(V)\cap\bC^i\subseteq \big(\big\{(x,y)\in \RR^{n+m}: |y-\bar y^h|< C(r)|\log\bE_1|^{1-1/m}\bE_1^{1/m}\text{ for some $h=1,\dots,Q$} \big\}\big)\,,
\end{equation}
provided that $\bE_1$ is small enough, depending upon $r$ and the geometric parameters, where  $\bar y^1,\dots,\bar y^Q\in\RR^n$ may depend on $i$.
If we prove \eqref{asdcas}, the proof can be concluded exactly as for the proof of Theorem \ref{height}, see after \eqref{mark}.

Now we prove \eqref{asdcas}. Let $q\in (m,m+1)$ to be chosen later depending upon $\bE_1$. Fix for the moment $i=1,\dots,M$, we suppress the subscript from $x_i$.
By Lemma \ref{vfedacssc}, we apply Theorem \ref{Lipapprox} in scale-invariant form with a fixed constant $\lambda\in (0,1)$ small enough, depending upon the geometric parameters, so that  we obtain the   $Q$-valued Lipschitz function $f:B_{\bar s}(x)\rightarrow\Iqs$ with associated set $K_\lambda\subseteq B_{\bar s}(x)$. We recall that by \eqref{lipen1}, as $\lambda$ is fixed depending upon the geometric parameters,
 \begin{equation}
     \int_{B_{\bar s}(x)}|D f|^q\le C_0\int_{\bC_{4\bar s}(x)} |\p_{\pi_0}-\p_{T_z V}|^q\,\dd\|V\|\le C_0 \bE_1\,.
 \end{equation}
 Notice that
 \begin{equation}
     \sup_{x\in B_{\bar s}(x) }\cG(f(x),f(0))\le C(r) (q-m)^{-1+1/q} \bigg(\int_{B_{\bar s}(x)}|D f|^q\bigg)^{1/q}\le C(r)(q-m)^{-1+1/q} \bE_1^{1/q}\,.
 \end{equation}
Indeed, set $g\defeq\cG(f(x),f(0))\in \mathrm{LIP}(B_{\bar s}(x))$ with $|\nabla g|^2=\sum_{j=1}^m|\partial_j g|^2\le|D f|^2$ a.e., which implies that $|\nabla g|^q\le |D f|^q$. Hence, the inequality follows by the proof of the Morrey’s inequality in \cite[Theorem 4.10 (i)]{EG15}, keeping track of the dependence on $q$ of the constants.
Therefore, 
 \begin{equation}\label{asdcas1}
\supp(V)\cap (K_\lambda\times\RR^n)\subseteq \big(\big\{(x,y)\in \RR^{n+m}: |y-\bar y^h|< C(r)(q-m)^{-1+1/q}\bE_1^{1/q}\text{ for some $h=1,\dots,Q$} \big\}\big)\,,
\end{equation}
for $\bar y^1,\dots,\bar y^Q\in\RR^n$ depending on $i$.
Now we optimize the choice of $q$. We choose $q=m+\delta\in (m,m+1)$,  where $\delta \defeq -\frac{1}{\log\bE_1}$, provided that $\bE_1$ is smaller than a geometric constant, notice that $1/q\ge 1/m-\delta/m^2$.
We therefore bound
\begin{equation}
    \begin{split}
        (q-m)^{-1+1/q}\bE_1^{1/q}&\le \delta^{-1+1/m-\delta/m^2}\bE_1^{1/m-\delta/m^2}\le C_0|\log\bE_1|^{1-1/m}\bE_1^{1/m}\le C_0|\log \bE_1|^{1-1/m}\bE_1^{1/m}\,.
    \end{split}
\end{equation}
This proves \eqref{asdcas} with the left-hand-side replaced by $\supp(V)\cap (K_\lambda\times\RR^n)$.

Now we deal with $\supp(V)\cap ((B_{\bar s}(x)\setminus K_\lambda)\times \RR^n)$. Recall \eqref{measureestimate}, which states that
\begin{equation}
    |B_{\bar s}(x)\setminus K_\lambda|\le C(r) \bE_1\eqdef N\bE_1\,.
\end{equation}
Set $\sigma\defeq (2\omega_m^{-1} N\bE_1)^{1/m}$ and consider any $B_\sigma(z)\subseteq B_{\bar s}(x)$. If $\bE_1$ is small enough (depending upon $r$ and the geometric parameters), balls of this type cover $B_{\bar s}(x)$. Hence, for $y\in B_{\bar s}(x)\setminus K_\lambda$, take one such ball $B_\sigma(z)\ni y$. By the measure estimate, there exists $y'\in B_{ \sigma}(z)\cap  K_\lambda$.
Hence, by \eqref{brvfeads}, if we set $f(y')=\sum_i Q_i\a{p_i}$,
\begin{equation}\label{asdcas2}
    \supp(V)\cap \bC_{\sigma}(z)\subseteq B_\sigma(z)\times \bigcup_i B_{C_0 \sigma}(p_i)\subseteq B_\sigma(z)\times \bigcup_i B_{C_0 \sigma+C(r)|\log\bE_1|^{1-1/m}\bE_1^{1/m}}(\bar y^h)\,.
\end{equation}
Then \eqref{asdcas} follows, as 
$C_0\sigma\le C(r) \bE_1^{1/m}\le C(r)|\log\bE_1|^{1-1/m}\bE_1^{1/m}$.    
\end{proof}
\begin{proof}[Proof of \eqref{heighatdensity} of  Theorem \ref{heightdue}] Of course, we can assume $r>1/2$. Let $ \bar \rho\defeq (r+1)/2$ and $\bar s\defeq (1-\bar \rho)/8=(1-r)/16$, we set $r_1\defeq \bar s$, $r_2=1-\bar s$ and assume that $\bE_1$ is small enough to have the conclusion of Lemma \ref{vfedacssc} with $\eta=1/4$, for some $Q'\in \{0,\dots,Q\}$. Notice that $\|V\|(\bC_{1/2})\ge \|V\|(\bB_{1/2})\ge Q\omega_m (1/2)^m$ by the monotonicity formula, so that $Q'=Q$.

We cover now $\bC_{\bar \rho}$ with $\bC^1,\dots,\bC^M$, where $M\le C(r)$, such that for every $i$, $\bC^i=\bC_{\bar s}(x_i)$, with $x_i\in B_{\bar \rho}$. Notice that $\bE(\bC_{4\bar s}(x_i))\le \frac{1}{\bar s^m}\bE_1$.
By  Lemma \ref{vfedacssc}, we can apply Proposition \ref{lipen} in scale-invariant form to every $\bC_{4\bar s}(x_i)$, provided that  $\bE_1$ is small enough, depending upon the geometric parameters and on $\bar s$, i.e.\ upon the geometric parameters and on $r$. Here, $\lambda\in (0,1)$ is assumed to be small, depending upon the geometric parameters, and fixed.

Hence, for every $i$, we have sets $K_{\lambda,i}\subseteq B_{\bar s}(x_i)$, satisfying \eqref{area}.
Taking into account \eqref{measureestimate}, we obtain that 
\begin{equation}
    \Big|B_{\bar \rho}\setminus\bigcup_i K_{\lambda,i}\Big|\le\sum_i |B_{s_i}(x_i)\setminus K_{\lambda,i}| \le C(r) \bE_1\,, 
\end{equation}
so that 
\begin{equation}
    \|V\|(\bC_{\bar\rho})\le  Q\omega_m \bar\rho^m+C_0 \int_{\bC_{\bar\rho}}|\p_{T V}-\p_{\pi}|^2\,\dd\|V\|+  C(r) \bE_1 \le Q \omega_m {\bar\rho}^m+ C(r)\bE_1\,.
\end{equation} 
  Hence, by \eqref{monotformula},
  \begin{equation}
        \int_{\bB_{\bar\rho}}\frac{| x^\perp|^2}{|x|^{m+2}}\,\dd\|V\|= \frac{\|V\|(\bB_{\bar\rho})}{\omega_m \bar\rho^m}-\Theta(0)\le C( r)\bE_1\,.
        \end{equation}
    Now the conclusion can be obtained following e.g.\ \cite[Section 1]{SpolAlm}, see in particular \cite[Lemma 1.7 and Lemma 1.8]{SpolAlm}. In particular, by (the proof of) \cite[Lemma 1.8]{SpolAlm}, choosing coordinates $(x,y)\in\RR^{m+n}$,
    \begin{equation}\label{bfvsdcvss}
        \int_{\bB_{\bar\rho}} |y|^2\,\dd\|V\|\le C(r)\bE_1\,.
    \end{equation}
    Also, by Theorem \ref{height}, we have that 
    \begin{equation}
    \supp(V)\cap\bC_{ \bar\rho}\subseteq  \RR^m\times \bar B_{C( r)\bE_1^{1/(2m)}}(0)\,,
    \end{equation}
    provided that $\bE_1$ is small enough, depending upon $r$ and the geometric parameters.
    Hence, if $\bE_1$ is small enough, depending upon $r$ and the geometric parameters, we have by \eqref{bfvsdcvss} that 
      \begin{equation}
    \int_{\bC_{(r+\bar\rho)/2}}|y|^2\,\dd\|V\|\le C(r)\bE_1\,.
    \end{equation}
    Now, the conclusion follows from the subharmonicity of $|y|^i$, as in \cite[Lemma 1.7]{SpolAlm}, see e.g.\ \cite[7.5 (6)]{AllardFirst} or \cite[Proposition 2.3]{DPGaSc}.
\end{proof}

\subsection{Improved Lipschitz approximation}
The improved height bound of Theorem \ref{heightdue} yields an improved version of the Lipschitz approximation of Theorem \ref{Lipapprox}.
\begin{proof}[Proof of Theorem \ref{Lipapproxstrong}]
    Theorem \ref{Lipapproxstrong} is proved exactly as   Theorem \ref{Lipapprox}, but relying on Theorem \ref{heightdue} in place of Theorem \ref{height}. The only thing to notice is that, if $\lambda$ is smaller than a geometric constant, then, for every $\alpha\in (0, \lambda)$ we have $|\log\alpha|^{1-1/m}\alpha^{1/m}\le |\log\lambda|^{1-1/m}\lambda^{1/m}$.
 \end{proof}
\subsection{Excess decay}\label{excdec}
From Theorem \ref{predecay}, Theorem \ref{decay} follows almost immediately. We are going to prove Theorem \ref{predecay} following the blueprint in \cite{DelellisAllard}.
\begin{thm}\label{predecay}
	Let $V$ be stationary in $\bB_5$. 
 Then, given any $\delta\in (0,1)$, there exists a threshold $\epsilon = \epsilon (\delta, Q,m,n) >0$ and an $\eta=\eta (\delta, Q, m, n)\in(0,1)$ such that, if 
 \begin{gather}
		\bE(V,\bB_5)<\epsilon\,,\\
  \mathcal{H}^m(\{x\in\supp(V)\cap \bB_5: \Theta(x)<Q\})<\epsilon\label{adfcas}\,,\\
    Q-\eps\le\frac{\|V\|(\bB_{1})}{\omega_m} \le\frac{\|V\|(\bB_{5})}{\omega_m 5^m} < Q+\epsilon\,,
  \end{gather}
  then 
	\begin{equation}\label{vfecdas}
		\bE(V,\bB_{5\eta})\le \eta^{2-2\delta}\bE(V,\bB_5)\,.
	\end{equation}
\end{thm}

To prove Theorem \ref{predecay}, we need a technical lemma which states that if a $Q$-valued function $f$ is almost harmonic, in a weak sense, and over a large set its support consists of a single sheet, then $f$ is close, in $L^2$, to a single-valued harmonic function counted with multiplicity $Q$.
\begin{lem}[Harmonic approximation] \label{harmapprox}Consider $B_1\subseteq\RR^m$. For any $\rho>0$, there exists a threshold $\epsilon_A=\epsilon_A (\rho, Q,m,n) >0$ with the following property. If a function $f\in W^{1,2}(B_1,\Iq(\RR))$ with $\D(f,B_1)\le 1$ satisfies 
	\begin{equation}\label{ass1}
		\bigg|\int_{B_1} \sum_{i=1}^Q\nabla \varphi\,\cdot\,{D f_i}\bigg|\le \epsilon_A \|\nabla\varphi\|_{L^\infty}\qquad\text{for every $\varphi\in C_c^1(B_1)$}
	\end{equation}
	and 
	\begin{equation}\label{ass2}
		\mathcal{L}^m(\{x\in B_1: f(x)\ne Q\a{\etaa\circ f}(x)\})\le\epsilon_A\,,
	\end{equation}
	there there exists an harmonic function $u\in W^{1,2}(B_1)$ with $\int_{B_1}|\nabla u|^2\le 1/Q$ such that 
	\begin{equation}\label{conc1}
		\int_{B_1} \cG(f,Q\a{u})^2\le \rho\,.
	\end{equation}
\end{lem}
	\begin{proof}
		We argue by contradiction, so we take $\rho>0$ and a sequence of functions $(f_j)_j\subseteq W^{1,2}(B_1,\Iq(\RR))$ with $\D(f_j,B_1)\le 1$, satisfying \eqref{ass1} and \eqref{ass2} with $2^{-j-1}$ in place of $\epsilon_A$, but such that \eqref{conc1} fails for every harmonic function $u\in W^{1,2}(B_1) $ with $\int_{B_1} |\nabla u|^2\le 1/Q$.
		We set $E_j\defeq \bigcup_{h\ge j}\{x\in B_1:f_h(x)\ne Q\a{\etaa\circ f_h}(x)\}$, notice that $\mathcal{L}^m(E_j)\le 1/2^j$,
		
		Set $\bar f_j\defeq \etaa\circ f_j\in W^{1,2}(B_1)$ with $\int_{B_1} |\nabla \bar f_j|^2\le 1/Q$. Indeed, if $f=\sum_i \a{f_i}$, then a.e.\ we have   
		\begin{equation}
			|D(\etaa\circ f)|^2=\frac{1}{Q^2} \Big|\sum_i D f_i\Big|^2\le \frac{1}{Q^2} Q \sum_i |D f_i|^2\,,
		\end{equation} 
		where we used Holder's inequality.
		By the Poincaré inequality, there exists a sequence $(c_j)_j\subseteq\RR$ such that $(\bar f_j-c_j)_j$ is bounded in $L^2(B_1)$.
		Now, using the notation $\big(\sum_i \a{P_i}\big)\ominus c\defeq\sum_i 	\a{P_i-c}$,  we compute
		\begin{equation}\label{cdsacsa}
\cG(f_j\ominus c_j,Q\a{0})\le \cG(f_j\ominus  c_j,Q\a{\bar f_j- c_j})+\cG(Q\a{\bar f_j- c_j},Q\a{0})=	\cG(f_j\ominus \bar f_j,Q\a{0})+\sqrt{Q}|\bar f_j-c_j|\,.
		\end{equation}
		Set $g_j\defeq \cG(f_j\ominus \bar f_j,Q\a{0})$.
		Recall that $g_j=0$ a.e.\ on $B_1\setminus E_1$, and moreover $\int_{B_1} |\nabla g_j|^2\le C_0 $ for every $j$. Hence, by the Poincaré inequality, we see that the first term at the right-hand-side of \eqref{cdsacsa} is bounded  in $L^2(B_1)$ uniformly in $j$. Also, the second term at the right-hand-side of \eqref{cdsacsa} is  bounded in $L^2(B_1)$ uniformly in $j$,  by the choice of $c_j$. Hence, we see that ($f_j\ominus c_j)_j$ is bounded in $L^2(B_1)$.  
		
		Therefore, by \cite[Proposition 2.11]{DSq} and the Sobolev embedding, up to passing to a non-relabelled subsequence, we have $f\in W^{1,2}(B_1,\Iq(\RR))$ and $\bar f\in W^{1,2}(B_1)$ with $\int_{B_1}|\nabla \bar f|^2\le 1/Q$ such   that  $f_j\ominus c_j\rightarrow f$ strongly in $L^2(B_1)$, and   $\bar f_j- c_j\rightarrow \bar f$ strongly in $L^2(B_1)$ and weakly in $W^{1,2}(B_1)$. Up to passing to a further, non-relabelled subsequence, we can assume that the convergence is also pointwise.   It then follows that $f=Q\a{\etaa\circ f}=Q\a{\bar f}$ a.e., as still $\mathcal{L}^m(E_j)\rightarrow 0$.
		
		Now we notice that, if  $\varphi\in C_c^1(B_1)$, we have for every $j$ that
		\begin{equation}\label{veafdcsca}
				\begin{split}
					Q\bigg|	\int_{B_1}\nabla\varphi\,\cdot\,\nabla \bar f_j\bigg|&\le Q\bigg|	\int_{B_1\cap E_j}\nabla\varphi\,\cdot\,\nabla \bar f_j\bigg|+\bigg|	\int_{B_1\setminus E_j}\sum_{i=1}^Q \nabla\varphi\,\cdot\,D(  (f_j)_i)\bigg|
				\\&	\le Q\bigg|	\int_{B_1\cap E_j}\nabla\varphi\,\cdot\,\nabla \bar f_j\bigg|+\bigg|	\int_{B_1}\sum_{i=1}^Q \nabla\varphi\,\cdot\,D(  (f_j)_i)\bigg| +\bigg|\int_{B_1\cap E_j}\sum_{i=1}^Q \nabla\varphi\,\cdot\,D(  (f_j)_i)\bigg|\,.	
				\end{split}
					\end{equation}
		To deal with the first term at the right-hand-side of \eqref{veafdcsca}, we notice that by Holder's inequality,
		\begin{align}
			\bigg|	\int_{B_1\cap E_j}\nabla\varphi\,\cdot\,\nabla \bar f_j\bigg|\le \|\nabla\varphi\|_{L^\infty}(\mathcal{L}^m(E_j))^{1/2}\bigg(\int_{B_1}|\nabla \bar f_j|^2\bigg)^{1/2}\le C_0 \|\nabla\varphi\|_{L^\infty}(\mathcal{L}^m(E_j))^{1/2}\,,
		\end{align}
		and the third term at the right-hand-side of \eqref{veafdcsca} is estimated similarly.
		Now, the remaining term converges to $0$ thanks to our  assumption in the contradiction argument. Hence, using also the weak convergence $\bar f_j\rightarrow f$ in $W^{1,2}(B_1)$, we see that 
		\begin{equation}
		\int_{B_1} \nabla\varphi\,\cdot\,\nabla \bar f=0\qquad\text{for every $\varphi\in C_c^1(B_1)$\,.}
	\end{equation}
	Therefore, $\bar  f$ is a classical harmonic function with $\int_{B_1}|\nabla \bar f|^2\le 1/Q$,
	and this is a contradiction as $f_j\ominus c_j\rightarrow  f=Q\a{\bar f}$ in $L^2(B_1)$.
	\end{proof}

	\begin{proof}[Proof of Theorem \ref{predecay}]
		In the course of the proof, several quantities will appear, we briefly describe the relations among these quantities. First, $\delta$ is as in the statement, and we recall that $m,n,Q$ are the geometric parameters. Then we will have $\epsilon,\eta, \lambda, \rho\in (0,1)$. The parameter $\lambda$ will be fixed depending only upon the geometric parameters. Then, at the end of the proof, we fix $\eta$ depending upon the geometric parameters and $\delta$. Then, we fix $\rho$ depending upon the geometric parameters, $\eta$ and $\delta$. Finally, we will fix $\epsilon$, depending upon the geometric parameters, $\rho$,  $\eta$, and $\delta$, hence,  ultimately depending upon the geometric parameters and $\delta$.
		
		Up to a rotation, we assume that $\bE(V,\bB_5)=\bE(V,\bB_5,\pi_0)$.
		First, we claim that if $\epsilon$ is smaller than a constant depending upon the geometric parameter, then $V\cap \{(x,y)\in \RR^{m+n}: |x|<4, |y|<1\}$ is stationary in $\bC_4$, $\frac{\|V\|(\bC_{4})}{\omega_m 4^m} < Q+1/2$, and $\supp(V\cap \{(x,y)\in \RR^{m+n}: |x|<4, |y|<1\})\subseteq \{(x,y)\in\RR^{m+n}:|y|<\rho\}$, where $\rho\in (0,1)$ will be specified at the end of the proof.
  We prove this by contradiction, so take a sequence of varifolds $(V_i)_i$ as in the assumptions of the statement for $\epsilon\searrow 0$, but not satisfying the claim above. By the Compactness Theorem of integral varifolds (e.g.\ \cite{Simon}),  we have that $V_i$ converges to the integral varifold $V_\infty$ (in the sense of varifolds), up to a non-relabelled subsequence. Hence,
  \begin{equation}
    Q\le \frac{\|V_\infty\|(\bB_{2})}{\omega_m 2^m}   \le \frac{\|V_\infty\|(\bB_{5})}{\omega_m 5^m} \le Q
  \end{equation}
   and $\bE(V_\infty,\bB_5,\pi_0)=0$. This implies that $V_\infty\cap \bB_5=Q\a{\pi_0}\cap \bB_5$, so that we obtain a contradiction as $\supp(V_i)\rightarrow \supp(V)$ in the Kuratowski sense in $\bB_5$ and $V_i\rightarrow V_\infty$ in the sense of varifolds in $\bB_5$. 
		
		From now on, we are going to work with the varifold $V\cap \{(x,y)\in \RR^{m+n}: |x|<4, |y|<1\}$ in place of $V$. As this coincides with  the original varifold in $\bB_1$, replacing the varifold is harmless to the aims of our theorem.  We recall the properties that $V$ satisfies: $\frac{\| V\|(\bC_4)}{\omega_m 4^m}<Q+1/2$,
  \begin{equation}\label{basso}
      \supp(V)\subseteq\{(x,y)\in \RR^{n+m}: |y|\le \rho\}\,,
  \end{equation}
\begin{equation}
	\bE_4=\bE( V,\bC_4,\pi_0)\le (5/4)^m\epsilon\quad\text{and}\quad\mathcal{H}^m(\{x\in\supp( V)\cap \bC_4: \Theta( V,x)<Q\})<\epsilon\,.
\end{equation} We are going to denote
$\bE(V,\bB_5,\pi_0)$ simply by $\bE$, so that $\bE_4\le (5/4)^m\bE$. Notice that if $\bE=0$, then there is nothing to show, so that we will assume that $\bE>0$.
\medskip\\\textbf{Step 1.} 
Fix $\lambda\in (0,1)$ small enough, depending upon the geometric parameters, so that we apply Theorem \ref{Lipapprox} and obtain a $Q$-valued Lipschitz function $f:B_1\rightarrow\Iqs$. Hence, we can, and will, absorb the factor $\lambda^{-1}$ in the constant that depends upon the geometric parameters.

We record the following two conclusions, given by  \eqref{lipen1} and \eqref{lipen2}, respectively:
\begin{equation}
	M\defeq \frac{\int_{B_1}|D f|^2}{\bE}\vee 1\le \frac{C_0\lambda^{-1}\bE_4}{\bE}\vee 1\le C_0\,,
\end{equation}
and
 \begin{equation}
    \bigg| \int_{B_1}\sum_{i=1}^Q\nabla \varphi\,\cdot\,D f_i\bigg|\le C_0 \|\nabla\varphi\|_{L^\infty}\frac{\bE_4}{\lambda^{-1}} \le C_0 \|\nabla\varphi\|_{L^\infty}\bE\qquad\text{for every $\varphi\in C_c^1(B_1)$}\,.
 \end{equation}
Also, by \eqref{densdvacs}, \eqref{measureestimate}, and \eqref{adfcas}, we deduce that
\begin{equation}\label{toharm3}
	\mathcal{L}^m (\{x\in B_1:f (x)\ne Q\a{\etaa\circ f}(x)\})\le C_0\lambda^{-1}\bE+ \epsilon\le C_0\epsilon\,.
\end{equation}
\medskip\\\textbf{Step 2.}
Thanks to the estimates of \textbf{Step 1}, we can apply Lemma \ref{harmapprox} to 
$\frac{f}{\sqrt{M\bE}}$, provided that $\epsilon$ is smaller than a quantity depending upon the geometric parameters and $\rho$ (which will be specified at the end of the proof), and obtain that there exists an harmonic function $u:B_1\rightarrow\R^n$ with 
\begin{equation}\label{cdscasd}
	\int_{B_1}\cG(f, Q\a{u})^2\le C_0 \rho \bE\quad\text{and}\quad
	\int_{B_1}|\nabla u|^2\le  M\bE\le C_0\bE\,.
\end{equation}
 We set $z_0\defeq (0,u(0))$ and $\bar\pi$ to be the affine plane parallel to the image of the differential of the map $x\mapsto (x,u(x))$ at $0$ and passing through $z_0$.
		\medskip\\\textbf{Step 3.}
We notice that, using \eqref{cdscasd}  
\begin{equation}
	\|u\|_{L^1(B_1)}\le C_0\bigg(\int_{B_1}\cG(Q\a{u},f)^2\bigg)^{1/2}+C_0 \bigg(\int_{B_1}\cG(f,Q\a{0})^2\bigg)^{1/2}
	\le C_0 (\rho\bE)^{1/2}+ C_0 \rho\,,
\end{equation}
where the last equality is due to \eqref{basso} with \eqref{vefadcsxac}. Hence, by the mean-value property for harmonic functions,
\begin{equation}\label{vacds}
		\dist(z_0,\pi_0)=|u(0)|\le C_0\| u\|_{L^1(B_1)}\le C_0 (\bE^{1/2}+\rho)\le \bE^{1/4}+\rho^{1/4}\,,
\end{equation}
provided that $\rho,\eps$ are smaller than a geometric constant, which we will from now on assume, and 
\begin{equation}
		|\p_{\pi_0^\perp}-\p_{\bar\pi^\perp}|\le C\sum_{j=1}^n|\nabla u_j(0)|\le C_0\| u\|_{L^1(B_1)}\le C_0 (\bE^{1/2}+\rho)\,.
\end{equation}
Also, for $\eta\in (0,1/2)$,
\begin{equation}\label{vefadcsx}
	\sup_{x\in B_\eta}|u(x)-u(0)-\nabla u(0)x|^2\le C_0\eta^4\int_{B_1}|\nabla u|^2\le C_0\eta^4 \bE\,,
\end{equation}
and this follows from standard facts concerning harmonic functions, e.g.\ \cite[Lemma 6.2]{DelellisAllard}, and \eqref{cdscasd}.

Now, for $z\in \supp(V)\cap \bB_{1/2}$, we estimate by the previous inequalities and \eqref{basso},
\begin{equation}
	\dist(z-z_0,\bar\pi)=|\p_{\bar \pi^\perp} (z-z_0)|\le C_0(\bE^{1/2}+\rho)\,,
\end{equation}
so that, by the previous inequality and \eqref{measureestimate},
\begin{equation}\label{comb1}
	\int_{\bB_{1/2}\setminus (K_\lambda\times \RR^n)} 	\dist(z-z_0,\bar\pi)^2 \,\dd\|V\|(z)\le C_0\bE (\bE+\rho^2)\,.
\end{equation}
For $\eta\in (0,1/10)$, using \eqref{fadc} of the proof of Proposition \ref{lipen} and the area formula as we did in the proof of Proposition \ref{lipen}  (i.e.\ partitioning $B_1$ in $(C_h)_h$),
\begin{equation}\label{comb2}
\begin{split}
	&	\int_{\bB_{5\eta}\cap (K_\lambda\times \RR^n)} \dist(z-z_0,\bar\pi)^2 \,\dd\|V\|(z)\le\int_{\bB_{5\eta}\cap (K_\lambda\times \RR^n)} |z-u(0)-\nabla u(0)\p_{\pi_0}(z)|^2\Theta(\Gamma_f,z) \,\dd\mathcal{H}^m(z)
		\\
		&\qquad \le C_0 \int_{\bB_{5\eta}\cap (K_\lambda\times \RR^n)\cap \Gamma_f} \cG(f(\p_{\pi_0}(z)),Q\a{u(0)+\nabla u(0)\p_{\pi_0}(z)})^2 \,\dd\mathcal{H}^m(z)
		\\
		&\qquad\le C_0 \int_{B_{5\eta}}\cG(f(x),Q\a{u(0)+\nabla u(0)x})^2\,\dd\mathcal{L}^m(x)
		\\
		&\qquad\le C_0 \int_{B_{5\eta}}\cG(f,Q\a{u})^2\,\dd\mathcal{L}^m+  C_0 \int_{B_{5\eta}}\cG(u(x),{u(0)+\nabla u(0)x})^2\,\dd\mathcal{L}^m(x)\le C_0\rho\bE+C_0\eta^{m+4}\bE\,,
\end{split}
\end{equation}
where for the last inequality we used \eqref{cdscasd} and \eqref{vefadcsx}.

All in all, combining \eqref{comb1} and \eqref{comb2}, noticing that if $\epsilon,\rho<\eta^4/16$, we have that $|z_0|=|u(0)|<\eta$ (see \eqref{vacds}), so that $\bB_{4\eta}(z_0)\subseteq\bB_{5\eta}$,
\begin{equation}\label{grbwvfecdasx}
	\frac{1}{\eta^{m+2}}\int_{\B_{4\eta}(z_0)}\dist(z-z_0,\bar\pi)^2\,\dd\|V\|\le( C_0\eta^{-m-2}\rho+C_0\eta^2+C_0\eta^{-m-2}(\bE+\rho^2))\bE\,.
\end{equation}
\medskip\\\textbf{Step 4.}
As, $\bB_\eta(0)\subseteq \bB_{2\eta}(z_0)$, using \cite[Proposition 4.1]{DelellisAllard}, we deduce from \eqref{grbwvfecdasx} that
\begin{equation}\label{casdcasd}
\begin{split}
		\bE(V,\bB_\eta,\bar \pi)&\le C_0\bE(V,\bB_{2\eta}(z_0),\bar\pi)\le ( C_0\eta^{-m-2}\rho+C_0\eta^2+C_0\eta^{-m-2}(\bE+\rho^2))\bE
		\\
		&=( C'\eta^{-m-4+2\delta}\rho+C'\eta^{2\delta}+C'\eta^{-m-4+2\delta}\bE)\eta^{2-2\delta}\bE\,,
\end{split}
\end{equation}
for some $C'$ geometric constant, independent of $\eps,\eta,\lambda,\rho$.

Now we choose $\eta\in (0,10)$ small enough (depending upon the geometric parameters and on $\delta$) so that $C'\eta^{2\delta}<1/100$, $\rho\in (0,\eta^4/16)$ small enough (depending upon the geometric parameters and on $\eta$ and $\delta$) so that $C'\eta^{-m-4+2\delta}\rho<1/100$, and finally $\epsilon\in(0,\eta^4/16) $ small enough (depending upon the geometric parameters, on $\rho$, on $\eta$, and on $\delta$) so that our use of Theorem \ref{Lipapprox} and Lemma \ref{harmapprox} (now $\rho$ is fixed) is justified, \eqref{basso} holds, and finally $C'\eta^{-m-4+2\delta}\bE<1/100$.
Hence,
\begin{equation}
		\bE(V,\bB_{5\eta/5},\bar \pi)\le 3/100 \eta^{2-2\delta}\bE\le (\eta/5)^{2-2\delta}\bE\,.
\end{equation}
This is exactly \eqref{vfecdas}, up to changing $\eta/5$ with $\eta$.
	\end{proof}

\begin{proof}[Proof of Theorem \ref{decay}]
	We let $\bE\defeq \bE(V,\bB_1)$.
	First, we iterate the scale-invariant form of Theorem \ref{predecay} to obtain that 
	\begin{equation}
		\bE(V, \bB_{\eta^k})\le (\eta^{2-2\delta})^k\bE=(\eta^k)^{2-2\delta}\bE\qquad\text{for every $k\in\N$}\,.
	\end{equation}
	Notice that we can continue iterating Theorem \ref{predecay} as $\bE(V,B_{\eta^{k}})\le \bE(V,B_{\eta^{k-1}})\le \dots\le  \bE(V,B_{1})<\epsilon$, as
	 \begin{equation}
		\frac{\|V\|(\bB_r)}{\omega_m r^m}<Q+\epsilon\qquad\text{for every $r\in (0,1)$}
	\end{equation}
	by \eqref{vrfecdsccas} and the monotonicity formula, and as \eqref{adfcas} is satisfied for every $r\in(0,1)$ by assumption. 
	Now take any $r\in (0,1)$ and let $k\in\N$ be the unique integer with $r\in (\eta^{k+1},\eta^{k}]$. Then, 
	\begin{equation}
		\bE(V,\bB_r)\le \bigg(\frac{\eta^k}{r}\bigg)^m	\bE(V,\bB_{\eta^k})\le \frac{1}{\eta^m} (\eta^k)^{2-2\delta}\bE\le \eta^{-m-2+2\delta}(\eta^{k+1})^{2+2\delta}\bE\le \eta^{-m-2+2\delta} r^{2-2\delta}\bE\,,
	\end{equation}
	which gives the claim, recalling that $\eta$ depends upon the geometric parameters and $\delta$.
\end{proof}

\end{document}